\documentclass[reqno,a4paper]{amsart}
\usepackage[utf8]{inputenc}
\usepackage[T1]{fontenc}
\usepackage[english]{babel}
\usepackage{lmodern}
\usepackage{amsmath,amssymb,amsfonts,amsthm,mathtools,bm}

\usepackage[shortlabels]{enumitem}
\usepackage{graphicx}
\usepackage{subcaption}
\usepackage[margin=2cm]{geometry}
\usepackage[hidelinks,breaklinks]{hyperref}
\usepackage[numbers]{natbib}
\setitemize{itemsep=0.3em}
\setenumerate{itemsep=0.3em}
\usepackage[dvipsnames]{xcolor}
\usepackage{comment}
\usepackage{bbm}

\makeatletter
\def\paragraph{\@startsection{paragraph}{4}%
  \z@{3pt}{-\fontdimen2\font}%
  {\normalfont\bfseries}}
\makeatother

\newtheorem{theorem}{Theorem}[section]
\newtheorem{lemma}[theorem]{Lemma}
\newtheorem{proposition}[theorem]{Proposition}
\newtheorem{corollary}[theorem]{Corollary}
\newtheorem{claim}[theorem]{Claim}

\theoremstyle{definition}

\newtheorem{definition}[theorem]{Definition}

\usepackage{cleveref}
\crefname{theorem}{Theorem}{Theorems}
\Crefname{theorem}{Theorem}{Theorems}
\crefname{lemma}{Lemma}{Lemmas}
\Crefname{lemma}{Lemma}{Lemmas}
\crefname{claim}{Claim}{Claims}
\Crefname{claim}{Claim}{Claims}
\crefname{definition}{Definition}{Definitions}
\Crefname{definition}{Definition}{Definitions}
\crefname{section}{Section}{Sections}
\Crefname{section}{Section}{Sections}

\newcommand{\eps}{\varepsilon}
\newcommand{\e}{\mathrm{e}}
\newcommand{\ex}{\mathrm{ex}}

\newcommand{\EM}{\mathrm{EM}}

\newcommand{\Po}{\mathrm{Po}}

\newcommand{\cF}{\mathcal{F}}

\usepackage{etoolbox}

\newcommand{\blfootnote}[1]{%
  \begingroup
  \renewcommand\thefootnote{}% Removes the number completely from the text
  \footnotetext{#1}% Places the text at the bottom
  \endgroup
}

\title{Long induced cycles in pseudorandom graphs}

\author[Diskin]{Sahar Diskin\textsuperscript{*}}

\author[Lichev]{Lyuben Lichev\textsuperscript{\textdagger}}

\author[Krivelevich]{Michael Krivelevich\textsuperscript{$\ddagger$}}

\author[Markbreit]{Itay Markbreit\textsuperscript{\S}}

\begin{document}

\begin{abstract}
We show that, for some absolute constants $c_1,c_2>0$, every $(n,d,\lambda)$-graph with $\lambda\le c_1 d$ contains an induced cycle of length at least $c_2n\log(d/\lambda)/d$. This is best possible up to the values of $c_1,c_2$. Our techniques include a multi-scale algorithmic analysis, an adapted depth-first exploration procedure, a link to percolation theory, and estimates for random row-and-column extraction in symmetric matrices.
\end{abstract}

\maketitle

\section{Introduction}

\blfootnote{\textsuperscript{*}D-MATH ETH Zurich, R\"amistrasse 101, 8092 Z\"urich, Switzerland. Email: \texttt{Sahardiskinmail@gmail.com}.}
\blfootnote{\textsuperscript{\textdagger}Institute of Statistics and Mathematical Methods in Economics, TU Wien, A-1040 Vienna, Austria.\\ Email: \texttt{lyuben.lichev@tuwien.ac.at}.}
\blfootnote{\textsuperscript{$\ddagger$}School of Mathematical Sciences, Tel Aviv University, Tel Aviv 6997801, Israel. Email: \texttt{krivelev@tauex.tau.ac.il}. Research supported in part by NSF-BSF grant 2023688.}
\blfootnote{\textsuperscript{\S}School of Mathematical Sciences, Tel Aviv University, Tel Aviv 6997801, Israel. Email: \texttt{markbreit@mail.tau.ac.il}.}

Long cycles in random graphs have been a main theme in probabilistic combinatorics literature for many decades.
Classic benchmark problems include the study of long cycles~\cite{Ana23,ADEKKL25,AEKP25,AF21,CDEK24,EJ08,E-DLW24,Fri86,LMS26,Rio14} and Hamiltonicity~\cite{AKS85,B83a,B84,BJ26,CE-DGKO24,DMM-CPS24,FF84,KS83,K76,M19,NST19,P76,RW92,RW94,SV08}.
A related popular and fruitful direction of research studies the length of a longest induced cycle.
In a pioneering work, Frieze and Jackson~\cite{FJ87} showed that binomial random graphs with fixed expected degree and random regular graphs with fixed degree $d\ge 3$ typically contain induced cycles of linear length.
Their result on binomial random graphs was improved by Łuczak~\cite{L93} and, independently, by Suen~\cite{S92}, proving that whenever $c>1$, $G(n,c/n)$ typically contains an induced cycle of length at least $(1+o(1))(\log c)n/c$. A simple first‑moment argument implies that the typical length of a longest induced cycle for large $c$ in $G(n,c/n)$ is at most $(1+o_c(1))\cdot 2(\log c)n/c$. Finally, this upper bound was shown to be asymptotically tight for large $c$ by Draganić, Glock, and Krivelevich~\cite{DGK22}.
%Their result on binomial random graphs was improved by Dragani{\'c}, Glock and Krivelevich~\cite{DGK22} who showed that $G(n,c/n)$ typically contains an induced cycle of length approximately $2(\log c)n/c$ for large $c$, which is asymptotically optimal.
In the other extreme when $c$ is close to 1, the same authors~\cite{DGK22b} showed that $G(n,c/n)$ typically contains an induced path of length $\Theta((c-1)^2n)$, which coincides with the order of the 2-core in the giant component up to the value of the hidden constant (see, for example, \cite[Lemma 2.16]{FK15}).
In the setting of random graphs on more general degree sequences, Enriquez, Faraud, M\'enard and Noiry~\cite{EFMN21} provided lower bounds for the length of a longest induced cycle by sharpening the analysis of the algorithm used in~\cite{FJ87}.
We direct the reader to the survey of Frieze~\cite{F19} for further results on cycles in random graphs.

In the context of the mentioned results, it is natural to wonder which features of the random graph determine the length of its longest cycle, induced or not.
In this direction, Dragani\'c, Montgomery, Munh\'a Correia, Pokrovskiy and Sudakov~\cite{DMM-CPS24} showed that sufficiently strong vertex-expansion implies Hamiltonicity, thus generalising many earlier results~\cite{B83a,FK02,zbMATH07945724,zbMATH05848528,KS03,P76}, a very recent work of Brada\v{c} and Janzer~\cite{BJ26} relaxed the expansion condition at the price of requiring approximate $d$-regularity where $d\ge \mathrm{polylog}(n)$.
While the mentioned criteria for Hamiltonicity need expansion on all scales, expansion on large, small or fixed scales is already enough to guarantee the existence of long paths and cycles.
For example, Krivelevich and Sudakov~\cite{KS13} have shown that graphs with large minimum degree $d$ percolated with parameter $p$ where $pd\ge 1+\eps$ typically contain a path of order $\Omega(d)$. 
This result was strengthened in the regime $pd\to \infty$ to containment of a cycle of length at least $d-o(d)$ by Krivelevich, Lee and Sudakov~\cite{KLS15} and by Riordan~\cite{Rio14}.
In graphs where sets up to size $k$ vertex-expand by a factor of $d$, Collares, Diskin, Erde and Krivelevich~\cite{CDEK24} showed that percolation with the same parameter $p$ where $pd\ge 1+\eps$ typically leaves behind a cycle of length $\Theta(kd)$, and Hollom~\cite{H25} later extended this result to graphs where only sets of a given size $k$ are required to expand.

A recent line of research studied the related problem of finding long induced paths and cycles in expanders, with particular focus on \emph{$(n,d,\lambda)$-graphs}, which are defined as follows: 
\begin{definition}
    Given a graph $G$, order its eigenvalues (that is, the eigenvalues of its adjacency matrix) by $\lambda_1 \geq \ldots \geq \lambda_n$. 
    The largest eigenvalue, $\lambda_1$, of any $d$-regular graph is easily seen to be $d$, sometimes referred to as the trivial eigenvalue of $G$. 
    A $d$-regular graph $G$ on $n$ vertices is said to be an \emph{$(n, d, \lambda)$-graph} if $\max\{|\lambda_2|,|\lambda_n|\}\le \lambda$.
\end{definition}
We refer the reader to the survey \cite{KS06} for an extensive discussion on $(n,d,\lambda)$-graphs and other notions of pseudorandom graphs.
%$d$-regular $n$-vertex graphs with eigenvalues $\lambda_1=d$ and $\lambda_2,\ldots,\lambda_n\in [-\lambda,\lambda]$, formally called \emph{$(n,d,\lambda)$-graphs}.
Dragani\'c and Keevash~\cite{DK25} showed that every $(n,d,\lambda)$-graph with $100\lambda\le d^{3/4}$ and $10d<n$ contains an induced path of length at least $n/64d$.
This result was strengthened by Diskin, Krivelevich, Markbreit and Zhukovskii~\cite{DKMZ25} who showed that, for suitable positive constants $\delta_1,\delta_2,\delta_3$, every $(n,d,\lambda)$-graph with $d\le\delta_3 n$ and $\lambda/d\le\delta_1$ contains an induced cycle of length at least $\delta_2 n/d$.
The authors also suggested that induced cycles of length $\Omega(n\log(d/\lambda)/d)$ may exist; indeed, independent sets of this order exist~\cite{AKS99,KS06}, and the bound is sharp up to the hidden constant in that case.
Our main result confirms this prediction, and guarantees the existence of induced cycles of length $\Omega(n\log(d/\lambda)/d)$ for the even broader family of \emph{$(n, (1\pm\gamma)d, \lambda)$-graphs}. A graph $G$ is called an \emph{$(n, (1\pm\gamma)d, \lambda)$-graph} (for some constant $\gamma\coloneq \gamma(n,d,\lambda)\ge 0$) if $G$ has $n$ vertices, each of them having degree in the interval $[(1-\gamma)d, (1+\gamma)d]$, and all eigenvalues of $G$ but the largest one have absolute values at most $\lambda$.

\begin{theorem}\label{th: main}
Fix a constant $\alpha\in (0,1)$. 
Then, there are constants $c_1\coloneqq c_1(\alpha),c_2\coloneqq c_2(c_1,\alpha),c_3\coloneq c_3(\alpha)>0$ such that, for all sufficiently large integers $n,d$ with $d\le \alpha n$, every $\left(n,\left(1\pm c_3(\lambda/d)^{1/3}\right)d,\lambda\right)$-graph $G$ with $\lambda\le c_1 d$ contains an induced cycle of length at least $c_2 n\log(d/\lambda)/d$.
\end{theorem}

A few comments are in order.
Since an $(n,d,\lambda)$-graph is also an $(n, (1\pm\gamma)d, \lambda)$-graph for every $\gamma\ge 0$, Theorem~\ref{th: main} applies to $(n,d,\lambda)$-graphs. 
In particular, Theorem~\ref{th: main} extends all previous results on a longest induced cycle~\cite{DKMZ25,DK25,FJ87} and on the largest independent set~\cite{AKS99,KS06} in $(n,d,\lambda)$-graphs.
Furthermore, for $(n,d,\lambda)$-graphs, Theorem~\ref{th: main} is actually optimal up to the value of the constants $c_1,c_2$: for example, for a wide range of $d\coloneq d(n)\ge 3$, the random regular graph $G(n,d)$ is known to be an $(n,d,3\sqrt{d})$-graph with high probability, see e.g.~\cite{zbMATH05037081,Pud15,zbMATH07751065,zbMATH07036340}, but its largest independent set typically has size at most $3n\log(d)/d$ (and thus, its longest induced cycle has at most twice this length).
A deterministic example matching the lower bound from Theorem~\ref{th: main} up to constant may also be found in \cite[Section~5]{DKMZ25}, though a slightly more careful analysis than the one presented there is required to make the $\log(d/\lambda)$ appear explicitly.

\paragraph{Plan of the paper.} The paper is structured as follows. In Section~\ref{s: outline}, we provide an outline of the proof of \Cref{th: main}. 
In Section~\ref{s: pre}, we set our notation, collect several useful lemmas and present some definitions for almost regular and pseudorandom graphs. 
%Section~\ref{s: generalisation of FHMV} is dedicated to strengthening a result on random induced subgraphs of pseudorandom graphs. 
Section~\ref{s: proof-main} is devoted to the proof of \Cref{th: main}. Finally, Section~\ref{s: conclusion} contains some concluding remarks and discusses possible extensions and directions for future research.

\section{Proof outline}\label{s: outline}
We present an outline of the proof of \Cref{th: main}. 
In the first part of the argument, we prove the theorem for the range $d^6\le n$. 
The range $d=o(n)$ is then handled by a reduction to the previous case via site percolation,
%bootstrapping argument using induced subgraph, 
and a separate deterministic argument for when $d=\Theta(n)$ and $\lambda$ is small completes the proof.

\subsection*{The sparse case}
We begin with the case $d^6\le n$. Fix constants
\[
0<c_0\ll c_1\ll \beta \ll \eps \ll 1,
\]
and write
\[
k=c_0\log(d/\lambda).
\]
Our goal is to construct $k$ vertex-disjoint induced path, each of length $\Omega(n/d)$, and to stitch them together into one induced cycle.

To this end, set $n_i=\e^{-(1+\eps)i}n$ and $ d_i=\e^{-(1+\eps)i}d$ for all $i\in\{0,\ldots ,k\}$. Starting from $G(0)=G$, we construct a nested sequence of induced subgraphs $G=G(0)\supseteq G(1)\supseteq\cdots\supseteq G(k)$ such that each $G(i)$ is a nearly-$d_{i}$-regular graph on $n_{i}$ vertices. 
We construct these graphs as follows. 
Given the graph $G(i-1)$ with the said properties,
%has already been constructed and is nearly-$d_{i-1}$-regular on $n_{i-1}$ vertices. We 
we choose a random set $Y_i\subseteq V(G(i-1))$ by retaining each vertex independently and with probability $p_i = (1+\eps)/d_{i-1}$.
We then delete the set $Y_i$ together with its neighbourhood from $G(i-1)$, and perform a clean-up procedure in the remaining graph to find an induced nearly-$d_i$-regular graph on $n_i$ vertices. 
Observe that, following the Poisson paradigm, deleting $Y_i$ and its neighbourhood expected to create the `correct' drop of the order/degree: the probability that a vertex survives after this procedure is about $\mathbb{P}(\text{Poisson}(d_{i-1}p_i)=0)\approx \e^{-(1+\eps)}$.

We show that typically the promised cleaning procedure can be done at all stages in \Cref{s: regular}.
%, we prove that this cleaning procedure can be done at all stages.  
First, Lemma~\ref{l: estimate on N[Yi]} shows that $V(G(i-1))\setminus \left(Y_i\cup N_{G(i-1)}(Y_i)\right)$ has roughly the expected size; we note that the hypothesis that $d$ is small is used here. 
Then, Lemma~\ref{l: good regular subgraph} shows that any large enough set in $G(i-1)$ contains a large nearly regular induced subgraph. Together, these give Lemma~\ref{l: validity}: with high probability, for all $i\in [0,k]$, the graphs $G(i)$ exist, have an order roughly $n_i$, are nearly-$d_i$-regular, and inherit the required expander-mixing property from the original graph. Indeed, starting with $V(G)$ (of size $n$), we iteratively remove the sets $Y_i\cup N_{G(i-1)}(Y_i)$, leaving approximately an $\e^{-(1+\eps)}$-fraction of the vertices. The expander-mixing property allows us to control the edge distribution of sets of size up to $O(\lambda n/d)$. Consequently, Lemmas~\ref{l: estimate on N[Yi]} and \ref{l: good regular subgraph} may be applied iteratively $\Theta(\log(d/\lambda))$ times. 

We show that typically, in each $G(i-1)[Y_i]$, we can find an induced path $P_i$ of length $\Omega(n/d)$.
%First, observe that deleting $Y_i$ together with its neighbourhood creates the `correct' density drop: the expected proportion of vertices neither in $Y_i$ nor in its neighbourhood is roughly $\e^{-(1+\eps)}$, matching the definitions of $n_i$ and $d_i$. Furthermore, 
Note that, by construction, 
%this deletion separates paths in different iterations: 
all future subgraphs $G(j)$ with $j\ge i$, and hence all future paths $P_{j+1}$, lie outside $Y_i\cup N_{G(i-1)}(Y_i)$, and thus vertices of $P_i$ have no neighbours in them.
We produce the said paths by analysing a DFS exploration on the graphs $G(i-1)[Y_i]$ similar to the one of \cite{DKMZ25}.
To this end, we prove that each random set $Y_i$ has two key properties. 
%needed to apply a DFS argument. These properties are proved conditionally on the history up to time $i-1$: once $G(i-1)$ is fixed and is good, the set $Y_i$ is a fresh random subset of it. 
%\lyuc{I commented out the conditional part here, as it looked a bit technical to me, but we can also have it back if you prefer.}
The first property is \emph{good vertex-expansion}: in Lemma~\ref{l: expansion}, we show that typically every set of size about $n_{i-1}/d_{i-1}$ expands in $G(i-1)$ almost as well as in a random $d_{i-1}$-regular graph. 
The second property is \emph{small excess}: namely, in \Cref{l: excess}, we show that we may typically delete a relatively small number of edges from $G(i-1)[Y_i]$ to form a forest. % (where $\ex(H)=|E(H)|-|V(H)|+\kappa(H)$ with $\kappa(H)$ being the number of connected components in the graph $H$). \lyuc{Can we just say that $O(\eps^3n_{i-1}/d_{i-1})$ edges can be deleted to make the graph $G(i-1)[Y_i]$ a forest?}
%In \Cref{l: excess}, we show that $\ex(G(i-1)[Y_i])=O(\eps^3n_{i-1}/d_{i-1})$ for every $i$. 
These two properties allow us to successfully run a modified DFS procedure described in \Cref{s: DFS}, and obtain an induced path $P_i\subseteq G(i-1)[Y_i]$ of length $\Theta(n_i/d_i)=\Theta(n/d)$.

It remains to stitch the paths. For each path $P_i$, call $P_i^-$ and $P_i^+$ its initial and terminal thirds, and let $W_i^-$ and $W_i^+$ be the sets of vertices outside $P_i$ which have exactly one neighbour in $P_i$, and that unique neighbour lies in $P_i^-$ or $P_i^+$, respectively. 
In \Cref{lem:many-leaves}, we show that both $W_i^-$ and $W_i^+$ are large. This follows from the almost optimal expansion of $P_i$: since the path has about $d_{i-1}|V(P_i)|$ 
%external 
neighbours, almost each one of them has a unique neighbour in $P_i$. 
%Then, we use the randomness of the construction once again. 
Note that a vertex in $W_i^\pm$ is a good candidate for connecting to another path, but it could create a chord if it has a neighbour in the random sets $Y_{i+1}, \ldots, Y_k$. To resolve this danger, Lemma~\ref{lem:restrict} uses the randomness of the construction to extract large subsets $Z_i^\pm\subseteq W_i^\pm$ such that no vertex of $Z_i^-\cup Z_i+$ has a neighbour in $Y_{i+1}\cup\cdots\cup Y_k$. 
The paths $P_i,P_{i+1}$ (indices seen modulo $k$) are patched via edges $e_i\in E_G(Z_i^+,Z_{i+1}^-)$
%paths of length 3 through $Z_i^+$ and $Z_{i+1}^-$ (which is non-adjacent to $Y_1,\ldots,Y_i$ thanks to the nested construction) 
found using the expander-mixing property: note that the endpoints of $e_i$ send edges neither to $Y_1,\ldots,Y_{i-1}$ (thanks to the nested construction) nor to $Y_{i+2},\ldots,Y_k$ (by the properties of $Z_i^\pm$). 
The patching edges $(e_i)_{i=1}^k$ are chosen to form an induced matching by minor alterations of $Z_i^+,Z_{i+1}^-$ at every stage.
This allows us to form an induced cycle of the desired length in the case $d^6\le n$.

%Finally, using the expander-mixing property, we greedily choose edges $e_i\in E_G(Z_i^+,Z_{i+1}^-)$ (with indices modulo $k$), so that the chosen edges from an induced matching. This allows us to stitch together a large portion of each path, forming an induced cycle of the desired length. This concludes the case of $d^6\le n$.

\subsection*{Bootstrapping to the full range of degrees}
The above argument proves the theorem for the sparse regime $d^6\le n$. To pass from this regime to large values of $d$, we consider \textit{random induced subgraphs} of the original graph, while preserving the pseudorandom structure of the original graph. 
To this end, in \Cref{thm:patched}, we prove a quantitative strengthening of a theorem of Ferber, Han, Mao, and Vershynin \cite[Theorem 6.2]{FHMV25} on random induced subgraphs of pseudorandom graphs. 
Roughly speaking, the theorem says that if $G$ is an $(n,d,\lambda)$-graph and $X\subseteq V(G)$ is a uniformly random subset of size $\sigma n$, then under suitable hypothesis on $n,d$ and $\lambda$, $G[X]$ is typically again pseudorandom, this time on $\sigma n$ vertices, with degrees about $\sigma d$, and second largest (in absolute value) eigenvalue about $\sigma \lambda$. 
Our sharpening relaxes the assumptions on $d$ and $\lambda$, thus allowing to iterate the result efficiently for denser (almost-)regular pseudorandom graphs with smaller spectral gap.
%The point of proving this extension is that the iteration described below requires a sufficiently efficient random-subgraph theorem for (almost-)regular pseudorandom graphs, with required hypothesis weak enough to survive repeated sampling. \lyuc{I think we should underline here that their theorem does not allow us to treat the sparse regime.}
The proof of \Cref{thm:patched} follows by improving the relevant operator-norm estimates for random submatrices.

We then use \Cref{thm:patched} to bootstrap on our result for $d^6\le n$ as follows. 
Suppose the main theorem has already been proven for all graphs whose degree is at most $n^s$ for some $s<1$. 
Fix a pseudorandom graph $G$ with somewhat larger degree $d$ on $n$ vertices. 
By choosing a random subset $X$ with density $\sigma$ so that the new degree $\sigma d$ is at most $(\sigma n)^s$, Theorem~\ref{thm:patched} allows us to deduce that the induced subgraph $G[X]$ is typically also an (almost-)regular pseudo-random graph, where a suitably long induced cycle exists. 
%This gives an induced cycle of length $\Omega\left(\frac{\sigma n\log(\sigma d/\sigma \lambda)}{\sigma d}\right)=\Omega(n\log(d/\lambda)/d)$, and this induced cycle is naturally also induced in the original graph. 
%\lyuc{Perhaps we can remove the previous sentence and stay less quantitative with the blue addition to the previous one?}
This one-step bootstrapping is formalised in Theorem~\ref{lem: iteration_1} and further iterated in Corollary~\ref{lem: concluding iterations}, thus gradually extending the range of possible degrees all the way to $d=o(n)$ when $\lambda$ is relatively large. 
%This is an example of one iteration of the bootstrapping step, which is formalised in \Cref{lem: iteration_1}. 
%In Corollary~\ref{lem: concluding iterations}, we show that it can be iterated several times, gradually extending the range of possible degrees, all the way to $d=o(n)$.
%\lyuc{Hmm, do we do it this way? The big lemma takes care also of smaller $d$ with also small $\lambda$. This is perhaps too technical to describe here, but perhaps there is a way to suggest that we are lying a tiny bit with this division.}

The final remaining case is when $d=\Theta(n)$ and $\lambda$ is small, which is handled separately in Lemma~\ref{lem: one by one}. 
There, starting from a vertex $v_0$, we build an induced path vertex by vertex via a deterministic algorithm which maintains simultaneously a large set of available extensions, and a large set of possible closing vertices near $v_0$ --- which ultimately allows us to close this path into an induced cycle. 
The key for ensuring the said property for the required number of steps is the expander-mixing property, which is especially powerful for dense graphs with good spectral expansion.
%The expander-mixing property shows that at each step there is a good next vertex, and eventually one can close the path into an induced cycle.  
Combining this with Corollary~\ref{lem: concluding iterations} and one final application of Theorem~\ref{thm:patched} proves
Theorem~\ref{th: main} for all $d\le \alpha n$.

\section{Notation and preliminaries}\label{s: pre}
First, we set up some notation. For an integer $m\ge1$, we denote $[m]=\{1,\ldots,m\}$. Given $x,y,z\in\mathbb R_{\ge 0}$, we write $z=x\pm y$ to say that $z\in[x-y,x+y]$.
We use usual hierarchy notation: for positive real numbers $x_1,\ldots, x_k$, we write $x_1\ll x_2\ll\cdots\ll x_k$ when, for every $i\in \{2,\ldots ,k\}$, $x_{i-1}$ has to be chosen sufficiently small as a function of $x_i, x_{i+1}\ldots, x_k$. 
We also use standard asymptotic notations.

For a graph $G$, we write $V(G)$ for the vertex set of $G$ and $E(G)$ for the edge set of $G$. 
%We call $|V(G)|$ the \emph{order} of $G$, and $|E(G)|$ the \emph{size} of $G$. %(Sahar): we never actually use size to mean number of edges.
Given sets $A,B\subseteq V(G)$, we denote by $e_G(A,B)$ the number of edges with one endpoint in $A$ and the other endpoint in $B$; here, edges with two endpoints in $A\cap B$ counted twice. 
We also abbreviate $e_G(A)\coloneqq e_G(A,A)/2=|E(G)|$. We denote by $N_G(A)$ the \emph{open neighbourhood} of $A$ in $G$,~that~is,
\[
N_G(A)=\{v\in V\setminus A:\text{ there exists }u\in A\text{ with }uv\in E(G)\},
\]
we write $N_G[A]=N_G(A)\cup A$ for the \emph{closed neighbourhood} of $A$ in $G$. 
For a vertex $v\in V$, we write $\deg_G(v)$ for the number of neighbours of $v$ in $G$. %\lyuc{Commented out the definition of $\deg(v,A)$ since it is equal to $e(v,A)$.}
%and a set $A\subseteq V$, we denote by $\deg_G(v,A)$ the number of neighbours of $v$ in $A$. When $A=V(G)$, 
When the underlying graph is clear from the context, we often omit the subscript $G$. 
Throughout the paper, we ignore rounding signs as long as
this does not affect the validity of our arguments.

\subsection{Concentration inequalities}

We start by stating a version of the Chernoff inequality for the hypergeometric distribution, see \cite[Theorem~2.10]{JLR00}.

\begin{lemma}\label{lem:chernoff_hyp}
Fix integers $m,n,N$ and define a random subset $S\subseteq [N]$ distributed uniformly over all subsets of size $n$.
Denote by $X$ the size of $S\cap [m]$.
Then,
\begin{align*}
\text{for all $t\ge 0$,}\qquad \max\{\mathbb{P}(X-\mathbb E X \ge t),\mathbb{P}(X - \mathbb E X \le -t)\} &\le \exp\left(-\frac{t^2}{2\mathbb E X + 2t/3}\right).
\end{align*}
\end{lemma}

We continue with the classic Bernstein's inequality, see \cite[Theorem 2.9.5]{Ver18}.

\begin{lemma}\label{lem:Bernstein}
Consider independent real-valued random variables $X_1, \dots, X_n$ such that, for all $i \in [n]$, such that $|X_i-\mathbb E X_i| \le M$ for some constant $M>0$. 
Write $X = \sum_{i=1}^n X_i$ and $\sigma^2 = \sum_{i=1}^n \mathrm{Var}[X_i]$. Then,
\begin{equation}
    \text{for all $t\ge 0$,}\qquad \max\{\mathbb{P}\left( X-\mathbb E X \ge t \right),\mathbb{P}\left( X-\mathbb E X \le -t \right)\} \le \exp \left( -\frac{t^2}{2\sigma^2 + 2Mt/3} \right).
\end{equation}
\end{lemma}

We also make use of a theorem of Hoeffding comparing the tails of sums of random subsets of $\{x_1,\ldots,x_n\}\subseteq \mathbb R$ obtained by subsampling with and without replacement, see~\cite[Section~6]{Hoe63}. 

\begin{lemma}\label{lem:Hoeffding}
Fix real numbers $x_1,\ldots,x_n$ in an interval $[a,b]$, and define $I$ to be a uniformly random subset of $[n]$ of size $m$. 
Denote by $X$ the sum of $(x_i)_{i\in I}$, and denote by $Y$ the sum of $m$ independent uniform random variables on $\{x_1,\ldots,x_n\}$. Then,
\[\text{for all $t\ge 0$,}\qquad \mathbb P(X-\mathbb E X\ge t)\le \mathbb P(Y-\mathbb E Y\ge t)\qquad \text{and}\qquad \mathbb P(X-\mathbb E X\le -t)\le \mathbb P(Y-\mathbb E Y\le -t).\]
\end{lemma}

For a constant $C>0$ and a domain $\Lambda=\Lambda_1\times\cdots\times\Lambda_m\subseteq\mathbb R^m$, a function $f:\Lambda\to\mathbb R$ is said to be \emph{$C$-Lipschitz} if changing one coordinate of its input changes its value by at most $C$.
We will use the following bounded difference inequality; see, for example, \cite[Theorem 3.9]{M98} and \cite[Corollary 6]{War16}.
\begin{lemma}\label{azuma}
Consider $p\in [0,1]$, a $C$-Lipschitz function $f:\{0,1\}^m\to\mathbb R$ and a vector $X=(X_1,\ldots,X_m)$ with independent Bernoulli$(p)$ entries. Then,
\[
\text{for every $t\ge0$,}\qquad \mathbb P\big(|f(X)-\mathbb E f(X)|\ge t\big)
\le 2\exp\left(-\frac{t^2}{2C^2mp+2Ct/3}\right).
\]
\end{lemma}

\subsection{Subtrees in graphs and random trees}
Next, we state an estimate on the number of subtrees of a graph rooted in a particular vertex; this bound appears, for example, as \cite[Lemma 2]{BFM98}.

\begin{lemma}\label{l: trees}
Consider a graph $G$ with maximum degree $\Delta$, a vertex $v\in V(G)$ and an integer $r\ge1$. Then, the number $t(v,r)$ of $r$-vertex subtrees of $G$ rooted at $v$ satisfies
\[
    t(v,r)\le \frac{r^{r-2}\Delta^{r-1}}{(r-1)!}.
\]
\end{lemma}

Finally, for $\alpha > 1$ we define $\bar{\alpha} < 1$ to be the unique solution (other than $\alpha$) of the equation $x\e^{-x}=\alpha \e^{-\alpha}$. It is known (see, for example, \cite{ErdosRenyi1960}) that 
\[
    \bar{\alpha}=\sum_{k=1}^\infty \frac{k^{k-1}}{k!}\left(\alpha\e^{-\alpha}\right)^k.
\]
Also, observe that 
\[
    \frac{n}{\bar{\alpha}}\sum_{k=1}^\infty \frac{(k-1)k^{k-2}}{k!} \left(\bar{\alpha}\e^{-\bar{\alpha}}\right)^{k}=\frac{\bar{\alpha}}{2}n,
\]
which can be seen from the fact that the LHS is asymptotically equal to the expected number of edges of $G(n,\bar{\alpha}/n)$ which lie on trees (see, for example, \cite{FKM04}). Together, these two facts imply that
\begin{equation}\label{eq: standard by now tool}
    \sum_{k=1}^\infty \frac{k^{k-2}}{k!} \left(\alpha\e^{-\alpha}\right)^{k}=\bar{\alpha}-\frac{\bar{\alpha}^2}{2}.
\end{equation}

\begin{comment}
Finally, given $\eps>0$, we denote by $y(\eps)$ the unique solution in $(0,1)$ to
\begin{equation}\label{eq: surv prob}
    y=1-\exp(-(1+\eps)y).
\end{equation}
We recall that it is the survival probability of a Galton-Watson tree with offspring distribution $\Po(1+\eps)$ (see, for example, \cite[Theorem~32.1]{FK15}), and satisfies $y(\eps)=2\eps-O(\eps^2)$.
\end{comment}

\subsection{Pseudorandom graphs}

The main tool in the analysis of $(n,d,\lambda)$-graphs is the celebrated \textit{expander mixing lemma} due to Alon and Chung~\cite{AC88}.
\begin{lemma}\label{l: eml}
For every $(n,d,\lambda)$-graph $G$ and vertex sets $A,B\subseteq V(G)$, $|e(A,B)-(d/n)|A||B||\le \lambda\sqrt{|A||B|}$.
\end{lemma}

We say that a graph on $n$ vertices is an \textit{$(n,(1\pm \gamma)d,\lambda)$-graph} if all the degrees of $G$ are between $(1-\gamma)d$ and $(1+\gamma)d$, and all of the eigenvalues of $G$ but the largest one have absolute value at most $\lambda$. The following claim asserts that $(n,(1\pm \gamma)d,\lambda)$-graphs satisfy the following generalisation of the expander mixing lemma.
\begin{lemma}[Corollary~4.3 of \cite{FHMV25}]\label{cor:EML-almost-regular}
For every $(n,(1\pm\gamma)d,\lambda)$-graph $G$ and sets $S,T\subseteq V(G)$, we have
    \begin{equation}    \label{eq: EML-almost-regular}
        \frac{(1-\gamma)^2d|S||T|}{(1+\gamma)n}-\frac{1+\gamma}{1-\gamma}\cdot \lambda \sqrt{|S||T|}
        \leq e(S,T)
        \leq \frac{(1+\gamma)^2d|S||T|}{(1-\gamma)n}+\frac{1+\gamma}{1-\gamma}\cdot \lambda \sqrt{|S||T|}.   
    \end{equation}
\end{lemma}

Next, we state and prove a generalisation of the Ramanujan spectral lower bound on $\lambda$.

\begin{lemma}\label{l: spectral-lower}
For every $(n,(1\pm \gamma)d,\lambda)$-graph $G$, we have 
\[
    \lambda^2\ge \frac{d((1-\gamma)n-(1+\gamma)^2 d)}{n-1}.
\]
In particular, if $d=o(n)$, then $\lambda\ge (1-o(1))\sqrt d$.
\end{lemma}
\begin{proof}
Denote by $\lambda_1,\lambda_2,\ldots,\lambda_n$ the eigenvalues of $G$, and by $d_1,d_2,\ldots,d_n$ the degree sequence of $G$. 
It is well-known that $\lambda_1$ lies between the average degree and the maximum degree of $G$ (see e.g. \cite[Proposition 3.1.2]{zbMATH05982708}), so $(1-\gamma)d\le \lambda_1\le (1+\gamma)d$. 
Furthermore, $\sum_{j=1}^n\lambda_j^2=\mathrm{tr}(A^2)=\sum_{j=1}^n d_i=(1\pm \gamma)dn$, and thus
\[
    \sum_{j=2}^n\lambda_j^2\ge (1-\gamma)dn-(1+\gamma)^2d^2=d((1-\gamma)n-(1+\gamma)^2d).
\]
As each summand on the left is at most $\lambda^2$, the claim follows.
\end{proof}

Finally, we introduce two important definitions. The first one formalises the type of `almost regularity' we require from our graphs.

\begin{definition}\label{def: reg}
Fix $\delta\in[0,1)$ and $D\ge1$. A graph $H$ is called $(\delta,D)$\emph{-nearly regular} if $H$ has no vertex of degree more than $(1+\delta)D$, and $H$ has at most $\delta |V(H)|$ vertices of degree less than $(1-\delta)D$.
\end{definition}

The last definition in this section provides a convenient notation for the pseudorandomness property~\eqref{eq: EML-almost-regular}.

\begin{definition}\label{def: EML}
For real numbers $\lambda,\gamma\ge 0$ and positive integers $d<n$, we say that a graph $H$ has property $\EM(n,(1\pm\gamma)d,\lambda)$ if \eqref{eq: EML-almost-regular} holds for all sets $S,T\subseteq V(H)$.
\end{definition}

Observe that if $H$ has property $\EM(n,(1\pm\gamma)d,\lambda)$, then so does every induced subgraph of $H$.

\section{Proof of \texorpdfstring{\Cref{th: main}}{Theorem 1.1}}\label{s: proof-main}
%(Sahar): Add buffer & explain structure.

\subsection{Setup and parameters}\label{s: setup}
We work with constants 
\[
    0<c_0\ll c_1\ll \beta\ll \eps\ll\alpha<1.
\]
To begin with, throughout Sections~\ref{s: setup}--\ref{s: proof_d<<n} we assume that
\begin{equation}\label{eq: d-small}
    d^6\le n,
\end{equation}
and prove the theorem for an $(n,d,\lambda)$-graph $G$, choosing $c_2\coloneqq (c_0,c_1,\beta,\eps)>0$ sufficiently small and setting $c_3=1$. Set 
\[
k=\left\lfloor c_0\log(d/\lambda)\right\rfloor
\]
and note that we may assume that $k\ge3$ (indeed, up to small and harmless change in constants, for $k<3$ the result follows from \cite[Theorem 1]{DKMZ25}).
We define
\[
    a=1+\eps,
    \qquad
    n_i=\e^{-ai}n,
    \qquad
    d_i=\e^{-ai}d,
    \qquad
    p_{i+1}=a/d_i
    \qquad \text{for all $0\le i\le k$}.
\]
Set $\delta_0=0$ and, for all $i\in [k]$, define
\[
    \delta_i=100\left(\delta_{i-1}+\left(\frac{\lambda}{d_{i-1}}\right)^{1/3}\right).
\]
In particular, $\lambda/d_i\le (\delta_{i+1}/100)^3$. Furthemore, note that $\lambda / d_i=\e^{ai}(\lambda/d)$ for all integers $i \in [0,k]$, and also
\begin{equation}\label{eq: bound on delta}
    \delta_i = \sum_{j=1}^i 100^j \left(\frac{\lambda}{d_{i-j}}\right)^{1/3} \le 100^{i+1}\left(\frac{\lambda}{d}\right)^{1/3}\e^{ai/3}
    \le 100^{2c_0\log(d/\lambda)}\left(\frac{\lambda}{d}\right)^{1/3}\le \left(\frac{\lambda}{d}\right)^{1/4}\le c_1^{1/4},
\end{equation}
where the second inequality uses that $k=\left\lfloor c_0\log(d/\lambda)\right\rfloor$, and the third inequality uses that $c_0$ is small enough. In particular, by choosing $c_1$ sufficiently small in terms of $\beta,\eps$ we may (and will) assume throughout that
\begin{equation}\label{eq: hierarchy}
    0\le\delta_i\ll \beta \ll \eps\ll1
    \qquad\text{for all integers }i\in [0,k].
\end{equation}
Finally, set $\gamma=(\lambda/d)^{1/3}$ and note that
\begin{equation}\label{eq:gamma}
100\gamma\le \min_{i\in [k]} \delta_i.
\end{equation}

\subsection{A regularisation step}\label{s: regular}
In this subsection, we inductively construct the sequence of induced subgraphs $G=G(0)\supseteq G(1)\supseteq\cdots\supseteq G(k)$ promised in the proof outline after proving a couple of auxiliary lemmas. 
The first lemma estimates the number of vertices which remain after deleting the closed neighbourhood of a random set.

\begin{lemma}\label{l: estimate on N[Yi]}
Fix an integer $i\in [0,k-1]$. Suppose that $G(i)$  is an induced subgraph of $G$ and is a $(\delta_i,d_i)$-nearly regular graph on $(1\pm2\delta_i)n_i$ vertices satisfying $\EM(n,(1\pm \gamma)d,\lambda)$. Let $Y_{i+1}\subseteq V(G(i))$ be obtained by retaining every vertex independently with probability $p_{i+1}$. Then, with probability at least $1-\exp(-\sqrt n)$,
\[
    |V(G(i))\setminus N_{G(i)}[Y_{i+1}]|
    =(1\pm\delta_{i+1}/10)\e^{-a}|V(G(i))|.
\]
\end{lemma}
\begin{proof}
Let $X=|V(G(i))\setminus N_{G(i)}[Y_{i+1}]|$. A vertex $v\in V(G(i))$ is counted by $X$ precisely when none of the vertices in $\{v\}\cup N_{G(i)}(v)$ are retained. Thus
\[
    \mathbb E X=\sum_{v\in V(G(i))}(1-p_{i+1})^{\deg_{G(i)}(v)+1}.
\]
Since $G(i)$ is $(\delta_i,d_i)$-nearly regular, all vertices have degree at most $(1+\delta_i)d_i$, and all but at most $\delta_i |V(G(i))|$ vertices have degree at least $(1-\delta_i)d_i$. Since $p_{i+1} = a/d_i\ll \delta_i\ll 1$, this gives
\[
    \mathbb E X = (1-p_{i+1})^{(1\pm \delta_i)d_i+1} |V(G(i))|\pm \delta_i |V(G(i))| =(1\pm4\delta_i)\e^{-a}|V(G(i))|
    =(1\pm\tfrac1{20}\delta_{i+1})\e^{-a}|V(G(i))|,
\]
where the last step uses $\delta_{i+1}\ge100\delta_i$. 

Next, note that adding or removing one vertex to/from $Y_{i+1}$ can change $X$ by at most $1+(1+\delta_i)d_i\le 2d_i$. 
By Lemma~\ref{azuma} with $m=|V(G(i))|=(1\pm2\delta_i)n_i$, $C=2d_i$, $p=p_{i+1}$ and $t=\frac1{20}\delta_{i+1}\e^{-a}|V(G(i))|$, we obtain
\[
    \mathbb P\left(|X-\mathbb E X|>\tfrac1{20}\delta_{i+1}\e^{-a}|V(G(i))|\right)
    \le \exp\left(-\Omega\left(\frac{\delta_{i+1}^2 n_i}{d_i}\right)\right).
\]
Now note that $\lambda\ge \sqrt d/2$ by Lemma~\ref{l: spectral-lower}, $d_i\le d$ by construction, and $d^6\le n$ by \eqref{eq: d-small}. Therefore,
\[
    \frac{\delta_{i+1}^2 n_i}{d_i} = \Omega\left(\left(\frac{\lambda}{d_i}\right)^{2/3}\frac{n}{d}\right)
     = \Omega\left(\frac{n}{d^{4/3}}\right)
    \ge n^{1/2}.
\]
completing the proof.
\end{proof}
Next, we show that, for every sufficiently large vertex subset $A$ of a $(\delta_i,d_i)$-nearly regular graph $G$, $G[A]$ contains a large $(\delta_{i+1},d_{i+1})$-nearly regular induced subgraph.

\begin{lemma}\label{l: good regular subgraph}
Fix an integer $i\in [0,k-1]$. Suppose that $G(i)$  is an induced subgraph of $G$ and is a $(\delta_i,d_i)$-nearly regular graph on $(1\pm2\delta_i)n_i$ vertices satisfying $\EM(n,(1\pm \gamma)d,\lambda)$. Consider a vertex set $A\subseteq V(G(i))$ such that
\[
    |A|\ge (1-\delta_{i+1}/5)\e^{-a}|V(G(i))|.
\]
Then, $G(i)[A]$ contains a $(\delta_{i+1},d_{i+1})$-nearly regular induced subgraph on $(1\pm2\delta_{i+1})n_{i+1}$ vertices.
\end{lemma}
\begin{proof}
By passing to a subset of $A$ and using that $100\delta_i\le \delta_{i+1}$, we may assume that
\begin{equation}\label{eq:|A|}
    |A|=(1-\delta_{i+1}/5)\e^{-a}|V(G(i))| =(1\pm\delta_{i+1}/3)n_{i+1}.    
\end{equation}
Define $U^+=\{v\in A:e_{G(i)}(v,A)>(1+\delta_{i+1}/2)d_{i+1}\}$.
If $U^+\neq\varnothing$, then
\begin{equation}\label{eq: lower_U}
    e(U^+,A)>(1+\delta_{i+1}/2)d_{i+1}|U^+|.
\end{equation}
On the other hand, the $\EM(n,(1\pm \gamma)d,\lambda)$ property gives
\begin{align}\label{eq: upper_U}\notag
    e(U^+,A)&\le \frac{(1+\gamma)^2d}{(1-\gamma)n}|U^+||A|+\frac{1+\gamma}{1-\gamma}\lambda\sqrt{|U^+||A|}
    \le \frac{(1+\gamma)^2}{1-\gamma}(1+\delta_{i+1}/3)d_{i+1}|U^+|+\frac{1+\gamma}{1-\gamma}\lambda\sqrt{2n_{i+1}|U^+|}\\
    &\le (1+\delta_{i+1}/25)(1+\delta_{i+1}/3)d_{i+1}|U^+|+2\lambda\sqrt{2n_{i+1}|U^+|},
\end{align}
where the last inequality follows by~\eqref{eq:gamma}. Combining \eqref{eq: lower_U} and \eqref{eq: upper_U} yields
\[
    \frac{\delta_{i+1}}{10}d_{i+1}|U^+|\le 2\lambda\sqrt{2n_{i+1}|U^+|}.
\]
Solving for $|U^+|$, we have
\begin{equation}\label{eq: high-set-bound}
    |U^+|<\frac{800\lambda^2 n_{i+1}}{\delta_{i+1}^2d_{i+1}^2}
    \le \delta_{i+1}^4n_{i+1},
\end{equation}
where the second inequality used that $\delta_{i+1}\ge 100(\lambda/d_i)^{1/3}$. 

Similarly, let
\[
    U^-=\{v\in A:e_{G(i)}(v,A)<(1-\delta_{i+1}/2)d_{i+1}\}.
\]
The same argument, now using the lower bound guaranteed by $\EM(n,(1\pm \gamma)d,\lambda)$, gives
\begin{equation}\label{eq: low-set-bound}
    |U^-|\le \delta_{i+1}^4n_{i+1}.
\end{equation}

Set $A'=A\setminus U^+$ and note that,
by combining \eqref{eq:|A|} and \eqref{eq: high-set-bound}, we have that $|A'|=(1\pm2\delta_{i+1})n_{i+1}$.
Moreover, the graph $G(i)[A']$ has maximum degree at most $(1+\delta_{i+1}/2)d_{i+1}<(1+\delta_{i+1})d_{i+1}$. It remains to count vertices whose degree in $A'$ is too small. Any vertex $v\in A'$ with
\begin{equation}\label{eq:degA'}
e_{G(i)}(v,A')<(1-\delta_{i+1})d_{i+1}    
\end{equation}
either belongs to $U^-$ or has at least $\delta_{i+1}d_{i+1}/2$ neighbours in $U^+$. Since $G(i)$ has maximum degree at most $(1+\delta_i)d_i$, the number of vertices of the latter type is at most
\[
    \frac{(1+\delta_i)d_i|U^+|}{\delta_{i+1}d_{i+1}/2}
    \le \delta_{i+1}^2n_{i+1}.
\]
Combining this~\eqref{eq:|A|} and \eqref{eq: low-set-bound} shows that $A'$ contains at most $(1-\delta_{i+1})|A'|$ vertices satisfying~\eqref{eq:degA'}, and thus $G(i)[A']$ is the desired $(\delta_{i+1},d_{i+1})$-nearly regular induced subgraph.
\end{proof}

\paragraph{The construction of $G(0),\ldots,G(k)$.}
Set $G(0)=G$ and suppose that $G(i)$ has been defined for some $i\in [0,k-1]$. 
Recall the random set $Y_{i+1}\subseteq V(G(i))$ retaining every vertex independently with probability $p_i$. 
If $G(i)\setminus N_{G(i)}[Y_{i+1}]$ contains a $(\delta_{i+1},d_{i+1})$-nearly regular induced subgraph on $(1\pm 2\delta_{i+1})n_{i+1}$ vertices, we choose one such subgraph and call it $G(i+1)$; otherwise, we set $G(i+1)=\emptyset$. Observe that by construction, $G(0),\ldots,G(k)$ form a nested sequence of induced subgraphs of $G$. Furthermore,

\begin{lemma}\label{l: validity}
With probability at least $1-\exp(-n^{1/3})$, for every integer $i\in [0,k]$, the induced graph $G(i)\subseteq G$ is a $(\delta_i,d_{i})$-nearly regular graph on $(1\pm2\delta_i)n_i$ vertices and satisfies $\EM(n,(1\pm \gamma)d,\lambda)$.
\end{lemma}

\begin{proof}
The assertion holds deterministically for $G(0)=G$. Assume it holds for $G(i)$. By Lemma~\ref{l: estimate on N[Yi]}, with probability at least $1-\exp(-\sqrt n)$, the set $A=V(G(i))\setminus N_{G(i)}[Y_{i+1}]$ has size at least $(1-\delta_{i+1}/5)\e^{-a}|V(G(i))|$. 
Then, Lemma~\ref{l: good regular subgraph} produces the desired $G(i+1)$. Since induced subgraphs inherit $\EM(n,(1\pm \gamma)d,\lambda)$, the claim follows by induction and a union bound over $k\le \log n$ stages.
\end{proof}

%(Sahar): Reached here.
In the rest of the proof, we work with the sequence $G(0),\ldots,G(k)$ of induced subgraphs of $G$, which satisfy the properties of Lemma~\ref{l: validity}. 
For every $i\in [k]$, we will abbreviate
\begin{equation}\label{eq: theta-def}
    \theta_i=\left(\frac{\lambda}{d_{i-1}}\right)^{\beta/8}.
\end{equation}
By the choice of $c_0$ and $c_1$, for every $i\in[k]$ we have
\begin{equation}\label{eq: theta-bounds}
    \delta_i,\theta_i\ll\beta\ll\eps,
    \qquad
    \theta_i\frac{n_{i-1}}{d_{i-1}} = \theta_i\frac{n}{d}\ge n^{4/5}.
\end{equation}
Indeed, the first inequality follows since $c_1\ll \beta$ and $\lambda/d_{i-1}\le \sqrt{\lambda/d}\le \sqrt{c_1}\ll \beta$, and the second inequality follows from Lemma~\ref{l: spectral-lower}, \eqref{eq: d-small} and the relation $\beta\ll 1$.

\subsection{Expansion of subsets of random sets}\label{s: expansion}
The following lemma estimates the probability that moderately large subsets of $Y_i$ expand well in $G(i-1)$. The proof follows ideas from \cite[Lemma 3.1]{K15}, adapted to the present nearly regular setting.

\begin{lemma}\label{l: expansion}
With probability $1-O(n^{-3})$, for every $i\in[k]$ and every subset $L\subseteq Y_i$ of size
\[
    \theta_i\frac{n_{i-1}}{d_{i-1}}\le |L|\le (1-\beta)\frac{n_{i-1}}{d_{i-1}},
\]
we have
\[
    |N_{G(i-1)}(L)|\ge (1-\beta)^2\left((1-\delta_i)d_{i-1}|L|-\frac{d_{i-1}^2|L|^2}{n_{i-1}}\right).
\]
\end{lemma}

\begin{proof}
Fix $i\in [k]$. 
We call a set $L\subseteq V(G(i-1))$ non-expanding if it violates the displayed inequality. We now bound from above the probability that $Y_i$ contains such a set. For any set $L$, we abbreviate $\ell\coloneqq |L|$.

Consider an ordered sequence $\tau=(v_1,\ldots,v_\ell)$ of distinct vertices. 
For all $j\in [\ell]$, write $L_j=\{v_1,\ldots,v_j\}$ and $N_j=N_{G(i-1)}(L_j)$. For $j\in \{2,\ldots, \ell\}$, given $v_1,\ldots,v_{j-1}$, call a vertex $v\in V(G(i-1))\setminus L_{j-1}$ \textit{bad} if it has fewer than
\begin{align}
 (1-\beta/4)\left(d_{i-1}-\frac{d_{i-1}(d_{i-1}+1)(j-1)}{n_{i-1}}\right)\label{eq: bad}   
\end{align}
neighbours in $V(G(i-1))\setminus(L_{j-1}\cup N_{j-1})$, and \textit{good} otherwise.

Suppose first that at most $\beta\ell/4$ vertices of $\tau$ are bad. Each good vertex increases the current open neighbourhood by at least the displayed quantity in \eqref{eq: bad}, except that vertices chosen later in the sequence may have been counted earlier as boundary vertices; the latter accounts for at most $\ell$ vertices. Since $\ell\le(1-\beta)n_{i-1}/d_{i-1}$, we get in this case
\begin{align}\label{eq: neigh bound}
    |N_{G(i-1)}(L_\ell)|\notag
    &\ge \left(\sum_{j=\beta\ell/4+1}^{\ell}(1-\beta/4)\left(d_{i-1}-\frac{d_{i-1}(d_{i-1}+1)(j-1)}{n_{i-1}}\right)\right)-\ell \\
    &\ge (1-\beta/4)^2d_{i-1}\ell-\ell-(1-\beta/4)\sum_{j=\beta\ell/4+1}^\ell\frac{d_{i-1}(d_{i-1}+1)(j-1)}{n_{i-1}}.
\end{align}
Now, observe that 
\[
    (1-\beta/4)\sum_{j=\beta\ell/4+1}^\ell\frac{d_{i-1}(d_{i-1}+1)(j-1)}{n_{i-1}}\le (1-\beta/4)\frac{d_{i-1}(d_{i-1}+1)}{n_{i-1}}\cdot \frac{\ell(\ell-1)}{2}\le  (1-\beta)^2\frac{d_{i-1}^2\ell^2}{n_{i-1}}.
\]
In addition, since $(1-\beta)^2\delta_id_{i-1}>1$, we have that 
\[
    (1-\beta/4)^2d_{i-1}\ell-\ell\ge (1-\beta)^2d_{i-1}\ell-(1-\beta)^2\delta_id_{i-1}\ell\ge (1-\beta)^2(1-\delta_{i})d_{i-1}\ell.
\]
Substituting these calculations back into \eqref{eq: neigh bound} yields
\begin{align*}
    |N_{G(i-1)}(L_\ell)|
    &\ge (1-\beta/4)^2d_{i-1}\ell-\ell-(1-\beta/4)\sum_{j=\beta\ell/4+1}^\ell\frac{d_{i-1}(d_{i-1}+1)(j-1)}{n_{i-1}}\\
    &\ge (1-\beta)^2\left((1-\delta_{i})d_{i-1}\ell-\frac{d_{i-1}^2\ell^2}{n_{i-1}}\right).
\end{align*}
Thus, every ordering of a non-expanding set has at least $\beta\ell/4$ bad positions.

We next bound the number of choices for a bad vertex at a fixed position $j$. Let $X=L_{j-1}\cup N_{j-1}$, let $U=V(G(i-1))\setminus X$, and let $B$ be the set of bad vertices. Since $G(i-1)$ is $(\delta_{i-1},d_{i-1})$-nearly regular and $|V(G(i-1))|=(1\pm2\delta_{i-1})n_{i-1}$, we have
\begin{align}
    |U|\ge (1-2\delta_{i-1})n_{i-1}-((1+\delta_{i-1})d_{i-1}+1)(j-1)\ge \left(\beta/2-3\delta_{i-1}\right)n_{i-1}\ge \beta n_{i-1}/3,\label{eq: u and beta}
\end{align}
where the second inequality uses that $j\le\ell\le(1-\beta)n_{i-1}/d_{i-1}$, and the last inequality uses that $\delta_{i-1}\ll \beta$. 
Further, note that
\begin{align}
    d_{i-1}-\frac{d_{i-1}(d_{i-1}+1)(j-1)}{n_{i-1}}\le \frac{d_{i-1}|U|}{n_{i-1}}+3\delta_{i-1}d_{i-1}\le (1+\beta/20)\frac{d_{i-1}|U|}{n_{i-1}}\label{eq: bounding inner term}
\end{align}
where the last inequality uses \eqref{eq: u and beta} and that $\delta_{i-1}\ll \beta$. 
Hence, by definition of the bad vertices,
\begin{align}
    e(B,U)&\le (1-\beta/4)\left(d_{i-1}-\frac{d_{i-1}(d_{i-1}+1)(j-1)}{n_{i-1}}\right)|B|\notag\\
    &\le (1-\beta/4)(1+\beta/20)\frac{d_{i-1}|U|}{n_{i-1}}|B|\le (1-\beta/10)\frac{d_{i-1}|U|}{n_{i-1}}|B|,\label{eq: upper bound on e(b,u)}
\end{align}
where the second inequality follows from \eqref{eq: bounding inner term}.
On the other hand, by the property $\EM(n,(1\pm \gamma)d,\lambda)$ of $G(i-1)$,
\[
e(B,U)\ge \frac{(1-\gamma)^2d|B||U|}{(1+\gamma)n}-\frac{1+\gamma}{1-\gamma}\lambda\sqrt{|B||U|},
\]
and thus by \eqref{eq: upper bound on e(b,u)},
\begin{align}\label{eq: lambda bound B U}
    \lambda\sqrt{|B||U|}\ge \frac{(1-\gamma)^3d_{i-1}|B||U|}{(1+\gamma)^2n_{i-1}}-\frac{1-\gamma}{1+\gamma}\cdot \left(1-\frac{\beta}{10}\right)\frac{d_{i-1}|U|}{n_{i-1}}|B|\ge \frac{\beta d_{i-1}}{20n_{i-1}}|B||U|,
\end{align}
where the last inequality holds since $\gamma\ll \beta$. By combining \eqref{eq: u and beta} and \eqref{eq: lambda bound B U}, we obtain
\begin{equation}\label{eq: bad-choices}
    |B|\le \left(\frac{20\lambda}{\beta d_{i-1}}\right)^2\frac{n_{i-1}^2}{|U|}\le \left(\frac{20}{\beta}\right)^3\left(\frac{\lambda}{d_{i-1}}\right)^2n_{i-1}.
\end{equation}

By using that $\tbinom{\ell}{\beta\ell/4}\le (4\e/\beta)^{\beta\ell/4}$, it follows that the number of ordered non-expanding $\ell$-tuples is at most
\[
    \binom{\ell}{\beta\ell/4}
    \bigg(\bigg(\frac{20}{\beta}\bigg)^3\left(\frac{\lambda}{d_{i-1}}\right)^2n_{i-1}\bigg)^{\beta\ell/4}
    n_{i-1}^{(1-\beta/4)\ell}
    \le
    \bigg(\bigg(\bigg(\frac{20}{\beta}\bigg)^4\left(\frac{\lambda}{d_{i-1}}\right)^2\bigg)^{\beta/4}n_{i-1}\bigg)^\ell.
\]
Multiplying by $p_i^{\ell}$
%$\left(\frac{1+\eps}{d_{i-1}}\right)^\ell$ 
and dividing by $\ell!$, the probability that $Y_i$ contains a non-expanding $\ell$-set is at most
\[
    \left[
    \left(\bigg(\frac{20}{\beta}\bigg)^4\left(\frac{\lambda}{d_{i-1}}\right)^2\right)^{\beta/4}
    \frac{(1+\eps)\e n_{i-1}}{d_{i-1}\ell}
    \right]^\ell.
\]
Since $\ell\ge\theta_i n_{i-1}/d_{i-1}$ and $\theta_i=(\lambda/d_{i-1})^{\beta/8}\ll \beta\ll 1$, the expression in the square brackets above is at most
\[
    (1+\eps)\e\cdot \left(\frac{20}{\beta}\right)^\beta \left(\frac{\lambda}{d_{i-1}}\right)^{3\beta/8}\le (1+\eps)\e\cdot \frac{20}{\beta} \left(\frac{\lambda}{d_{i-1}}\right)^{3\beta/8}\le 100 \frac{\theta_i^3}{\beta}\le \frac{1}{2}.
\]
%which is at most $1/2$ since $\theta_i\ll \beta\ll 1$. 
Summing over all admissible $\ell$ gives a failure probability at most
\[
    \sum_{\ell\ge \theta_i n_{i-1}/d_{i-1}}2^{-\ell}=O(n^{-4})
\]
by \eqref{eq: theta-bounds}, and the lemma follows by a union bound over $i\in [k]$.
\end{proof}

\subsection{Estimating the excess of random induced subgraphs}\label{s: excess}
The \emph{excess} of a graph $F$ is
\[
    \ex(F)=|E(F)|-|V(F)|+\kappa(F),
\]
where $\kappa(F)$ denotes the number of connected components of $F$. 

Before turning to estimate the excess of each $G(i-1)[Y_i]$, we require the following auxiliary result.

\begin{lemma}\label{l: small non-expand}
With probability $1-O(n^{-3})$, for every $i\in[k]$, there are at most $\theta_in_{i-1}/d_{i-1}$ vertices in components $F$ of $G(i-1)[Y_i]$ satisfying
\[
    |V(F)|<\theta_i\frac{n_{i-1}}{d_{i-1}}
    \qquad\text{and}\qquad
    |N_{G(i-1)}(F)|\le (1-\beta)^3d_{i-1}|V(F)|.
\]
\end{lemma}

\begin{proof}
Work on the event from Lemma~\ref{l: expansion}. Suppose that the conclusion fails for some fixed $i$. Let $\cF$ be the family of components $F$ of $G(i-1)[Y_i]$ satisfying the two displayed conditions. Since every member of $\cF$ has order less than $\theta_in_{i-1}/d_{i-1}$, we may choose a subfamily $\cF'\subseteq\cF$ such that, writing $U=\bigcup_{F\in\cF'}V(F)$,
\[
    \theta_i\frac{n_{i-1}}{d_{i-1}}\le |U|\le 2\theta_i\frac{n_{i-1}}{d_{i-1}}.
\]
The components in $\cF'$ are pairwise non-adjacent in $G(i-1)[Y_i]$, and therefore
\[
    |N_{G(i-1)}(U)|\le \sum_{F\in\cF'}|N_{G(i-1)}(F)|\le (1-\beta)^3d_{i-1}|U|.
\]
On the other hand, since $|U|$ lies in the range of Lemma~\ref{l: expansion},
\begin{align*}
    |N_{G(i-1)}(U)|
    &\ge (1-\beta)^2\left((1-\delta_i)d_{i-1}|U|-\frac{d_{i-1}^2|U|^2}{n_{i-1}}\right) \\
    &\ge (1-\beta)^2(1-\delta_i-2\theta_i)d_{i-1}|U|.
\end{align*}
As $\delta_i,\theta_i\ll\beta$, this is larger than $(1-\beta)^3d_{i-1}|U|$, a contradiction.
\end{proof}

\begin{lemma}\label{l: excess}
With probability $1-O(n^{-3})$, for every $i\in[k]$,
\[
    \ex(G(i-1)[Y_i])=O(\eps^3n_{i-1}/d_{i-1}).
\]
\end{lemma}

\begin{proof}
Fix $i\in [k]$. Since $G(i-1)$ is a $(\delta_{i-1},d_{i-1})$-nearly regular graph on $(1\pm2\delta_{i-1})n_{i-1}$ vertices, we have $e(G(i-1))=(1+O(\delta_{i-1}))d_{i-1}n_{i-1}/2$.

Moreover, adding or removing a single vertex to/from $Y_i$ can change the value of $|V(G(i-1)[Y_i])|$ by $1$, and of $e(G(i-1)[Y_i])$ by at most $2d_{i-1}$. 
Applying Lemma~\ref{azuma} with error term $\eps^3n_{i-1}/d_{i-1}$ gives
\[\mathbb P\bigg(|V(G(i-1)[Y_i])|\neq (1\pm \eps^3)p_in_{i-1}\bigg)\le \exp\bigg(-\frac{(\eps^3 n_{i-1}/d_{i-1})^2}{2p_in_i +2(\eps^3 n_{i-1}/d_{i-1})/3}\bigg),\]
and similarly
\[\mathbb P\bigg(e(G(i-1)[Y_i])\neq (1\pm \eps^3)p_i^2e(G(i-1))\bigg)\le \exp\bigg(-\frac{(\eps^3 p_i^2 e(G(i-1)))^2}{2(2d_{i-1}^2)e(G(i-1))p_i^2 +2\cdot 2d_{i-1}(\eps^3 p_i^2 e(G(i-1)))/3}\bigg).\]
Further using that $d^6\le n$ and $p_i^2 e(G(i-1)) = \Theta(n_{i-1}/d_{i-1}) = \Theta(n/d)$, with probability $1-O(n^{-3})$,
\begin{align}
|V(G(i-1)[Y_i])|=(1+O(\eps^3))\frac{(1+\eps)n_{i-1}}{d_{i-1}}\qquad \text{and}\qquad |E(G(i-1)[Y_i])|=(1+O(\eps^3))\frac{(1+\eps)^2n_{i-1}}{2d_{i-1}}.\label{eq: vertices-edges-b}
\end{align}

It remains to bound the number of components. By Lemma~\ref{l: small non-expand}, the number of components of order less than $\theta_{i} n_{i-1}/d_{i-1}$ and vertex-neighbourhood of size at most $(1-\beta)^3d_{i-1}|V(F)|$ is at most $\theta_i n_{i-1}/d_{i-1}$.
Also, \eqref{eq: theta-bounds} implies $\theta_i n_{i-1}/d_{i-1}\ge \theta_i^{-1}$. 
Combining this with the first part of \eqref{eq: vertices-edges-b} shows that there are $O(\theta_i n_{i-1}/d_{i-1})$ components of order at least $\theta_i^{-1}$. It remains to estimate the number $\kappa_i'$ of components $F$ of $G(i-1)[Y_i]$ with
\[
    |V(F)|<\theta_i^{-1}
    \qquad\text{and}\qquad
    |N_{G(i-1)}(F)|>(1-\beta)^3d_{i-1}|F|.
\]
For a vertex $v$ and an integer $r\ge1$, Lemma~\ref{l: trees} bounds the number of $r$-vertex trees in $G(i-1)$ rooted at $v$ by
\[
    \frac{r^{r-2}((1+\delta_{i-1})d_{i-1})^{r-1}}{(r-1)!}.
\]
After summing over the choice of a root and dividing by $r$, every possible vertex set $S$ of a component is counted at most once for each tree spanning $S$. If such a tree $T$ has at least $(1-\beta)^3rd_{i-1}$ vertex-neighbours in $G(i-1)$, then the probability that its vertex set spans a component of $G(i-1)[Y_i]$ is at most
\[
    \left(\frac{1+\eps}{d_{i-1}}\right)^r\left(1-\frac{1+\eps}{d_{i-1}}\right)^{(1-\beta)^3rd_{i-1}}.
\]
Consequently, summing over all $r\ge1$ for an upper bound,
\begin{align}
    \mathbb E\kappa_i'
    &\le \sum_{r\ge1} n_{i-1}\frac{r^{r-2}((1+\delta_{i-1})d_{i-1})^{r-1}}{r!}
    \left(\frac{1+\eps}{d_{i-1}}\right)^r\left(1-\frac{1+\eps}{d_{i-1}}\right)^{(1-\beta)^3rd_{i-1}}\nonumber \\
    &\le \frac{n_{i-1}}{d_{i-1}}\sum_{r\ge1}\frac{r^{r-2}}{r!}
    \big((1+O(\delta_{i-1}))(1+\eps)\big)^r
    \exp\big(-(1+\eps)(1-O(\beta))r\big).\label{eq: kappa-prime-expectation}
\end{align}
Using \eqref{eq: standard by now tool} and applying a (by now) standard tree-function estimate (see, for example, \cite[Section 5]{DK21} and \cite[p.~8]{FKM04}), we can conclude
\begin{equation}\label{eq: component-expansion-estimate}
    \mathbb E\kappa_i'
    \le \left(\frac12-\frac{\eps^2}{2}+O(\eps^3)+O(\beta)+O(\delta_{i-1})\right)\frac{n_{i-1}}{d_{i-1}}.
\end{equation}
Moreover, changing one vertex-exposure can change $\kappa_i'$ by at most $2\theta_i^{-1}+1\le 3n^{4/5}d/n$ (where the last inequality follows from~\eqref{eq: theta-bounds}); applying Lemma~\ref{azuma} with $t=\eps^3n_{i-1}/d_{i-1}$ gives $\kappa_i'\le \mathbb E\kappa_i'+O(\eps^3n_{i-1}/d_{i-1})$ with probability $1-O(n^{-3})$.

Combining the above and using that $\theta_i,\delta_{i-1},\beta\ll\eps$, we obtain with probability $1-O(n^{-3})$ that
\begin{align*}
    \ex(G(i-1)[Y_i])
    &\le \left(\frac{(1+\eps)^2}{2}-(1+\eps)+\frac12-\frac{\eps^2}{2}+O(\eps^3)\right)\frac{n_{i-1}}{d_{i-1}} \\
    &=O(\eps^3n_{i-1}/d_{i-1}).
\end{align*}
Together with the union bound over the $k$ possible values of $i$, this proves the lemma.
\end{proof}

\subsection{DFS procedure}\label{s: DFS}
For each $i\in[k]$, we construct a long induced path $P_i$ in $G(i-1)[Y_i]$ via a modified depth-first search procedure. The algorithm maintains four sets of vertices.
\begin{itemize}
    \item $A$ is the set of available vertices not yet processed by the algorithm. Initially, $A=V(G(i-1))$.
    \item $S\subseteq Y_i$ is the stack. The vertices in $S$ are vertices in $Y_i$ explored according to a last-in-first-out rule, and $S$ will always span an induced path.
    \item $W$ is the set of 'wasted' vertices. A vertex moves to $W$ if it is explored during the algorithm and does not enter $Y_i$. Initially, $W=\emptyset$.
    \item $R\subseteq Y_i$ is the set of 'reservoir' vertices. A vertex from $Y_i$ moves to $R$ only after it is revealed to have at least two neighbours in $S\cup T$.
    \item $T\subseteq Y_i$ is the set of 'trash' vertices. A vertex from $Y_i$ moves to $T$ when it is removed from the top of the stack.
\end{itemize}

Fix an arbitrary ordering $\sigma$ of $V(G(i-1))$. During the process, we expose one vertex at a time and \textit{reveal} whether it falls into $Y_i$ (which happens with probability $p_i$). One step of the DFS process goes as follows.
\begin{enumerate}
    \item If $S=A=\emptyset$, terminate. If $S=\emptyset$ and $A\ne\emptyset$, let $v$ be the first vertex of $A$ in the ordering $\sigma$, and reveal whether $v\in Y_i$. If $v\in Y_i$, move $v$ from $A$ to $S$; otherwise, move $v$ from $A$ to $W$. Then repeat.
    \item If $S\ne\emptyset$, let $v$ be the top vertex in the stack.
    \begin{enumerate}
        \item If $e(v,A)=0$, move $v$ to $T$ and then go to (1); otherwise, go to (b).
        \item Denote by $u$ the first vertex in $A$ with respect to $\sigma$ adjacent to $v$, reveal whether $u\in Y_i$ and delete $u$ from $A$.
        %If $u\notin Y_i$, repeat after moving the maximal terminal segment $W\subseteq S$ with $e_{G(i-1)}(W,A)=0$ from $S$ to $T$.
        If $u\in Y_i$ and $e(u,(S\cup T)\setminus\{v\})\ge1$, move $u$ to $R$. If not, place $u$ on top of $S$ and repeat.
        %In both cases, move the maximal terminal segment $W\subseteq S$ with $e_{G(i-1)}(W,A)=0$ from $S$ to $T$, and repeat.
    \end{enumerate}
\end{enumerate}
Some remarks are in place. First, note that the definition of $R$ requires a stronger property than needed: for example, a vertex adjacent to the top of the stack $S$ and with a neighbour in the trash $T$ may well be pushed in the stack while preserving the property that $G(i-1)[S]$ is an induced path. 
We define $R$ as above because this way the vertices there contribute to the excess of $G(i-1)[Y_i]$ in the following sense: every vertex $u$ entering $R$ is adjacent to the top of the stack $v$ and to another already explored vertex $w$, and $v,w$ belong to the same component revealed by the DFS exploration process.
Thus, adding $u$ would create a cycle, and the inequality $|R|\le\ex(G(i-1)[Y_i])$ follows by charging each reservoir vertex at the moment it is exposed.
Second, observe that every step reveals exactly one previously unrevealed vertex, and therefore the number of steps is at most $|V(G(i-1))|$. 
Third, if we denote by $v$ the vertex at the top of the stack at some moment, since a vertex is placed on the stack only if it has no revealed neighbour in $(S\cup T)\setminus\{v\}$, the stack spans an induced path at all times. 
Finally, after a vertex enters $T$, it has no neighbours in the current available set $A$, and this remains true thereafter.

\begin{lemma}\label{lem:long}
With probability $1-O(n^{-2})$, for every $i\in[k]$, at time $s_i=\left\lfloor\eps n_{i-1}/20\right\rfloor$ of the DFS process on $G(i-1)$, the stack $S$ contains at least $\eps^3n_{i-1}/(100d_{i-1})$ vertices.
\end{lemma}

\begin{proof}
Fix $i\in[k]$. 
As the first $s_i$ steps of the DFS reveal $s_i$ independent Bernoulli$(p_i)$ variables, by Lemma~\ref{lem:chernoff_hyp}
\begin{equation}\label{eq: dfs-binomial}
    |S\cup R\cup T|=(1\pm\eps^2)s_ip_i\le \eps n_{i-1}/(10d_{i-1})
\end{equation}
with probability $1-O(n^{-3})$, and recall that the event 
\[|R|\le \ex(G(i-1)[Y_i])=O(\eps^3n_{i-1}/d_{i-1})=O(\eps^2s_ip_i)\]
holds with the same probability by Lemma~\ref{l: excess}. 
By a union bound, these two events hold for every $i\in [k]$ with probability $1-O(n^{-2})$; we assume these properties in what follows.
%We condition on this event as well as on the event 
%\[\ex(G(i-1)[Y_i])=O(\eps^3n_{i-1}/d_{i-1})=O(\eps^2s_i\lyu{p_i})\] guaranteed with the same probability by Lemma~\ref{l: excess} for every $i\in [k]$.
%, so $|R|=O(\eps^3n_{i-1}/d_{i-1})=O(\eps^2s_i\lyu{p_i})$.

Assume for a contradiction that, at time $s_i$, we have $|S|<\eps^2 s_ip_i/4$. 
Then, \eqref{eq: dfs-binomial} implies
\[
    |T|\ge (1-O(\eps^2))s_ip_i\ge \eps n_{i-1}/(25d_{i-1}).
\]
By combining the latter display with \eqref{eq: dfs-binomial}, $|T|$ is between $\theta_in_{i-1}/d_{i-1}$ and $(1-\beta)n_{i-1}/d_{i-1}$ and hence Lemma~\ref{l: expansion} applies with $L=T$.

Next, as already noted, every neighbour of $T$ has already been revealed by time $s_i$, and thus $|N_{G(i-1)}(T)|\le s_i$. On the other hand, Lemma~\ref{l: expansion} gives
\begin{align*}
    |N_{G(i-1)}(T)|
    &\ge (1-\beta)^2\left((1-\delta_i)d_{i-1}|T|-\frac{d_{i-1}^2|T|^2}{n_{i-1}}\right)\\
    &\ge (1-\beta)^2(1-\eps/5)d_{i-1} s_ip_i
    \ge (1+\eps/2)s_i,
\end{align*}
where we used $p_i=(1+\eps)/d_{i-1}$ and $s_i/n_{i-1}\le \eps/10$. This contradicts $|N_{G(i-1)}(T)|\le s_i$, showing that
\[
    |S|\ge \frac{\eps^2}{4}s_ip_i\ge \frac{\eps^3}{100}\frac{n_{i-1}}{d_{i-1}}.\qedhere
\]
%A union bound over $i\in[k]$ completes the proof.
\end{proof}

Assuming Lemma~\ref{lem:long} and recalling the relation $n_{i-1}/d_{i-1}=n/d$, for each $i\in[k]$, denote by $P_i$ the induced path consisting of the first
\begin{equation}\label{eq: path-length}
    r\coloneq \left\lfloor\frac{\eps^3}{100}\frac{n}{d}\right\rfloor
\end{equation}
vertices of the stack at time $s_i$ in the DFS exploration of $G(i-1)$, which in turn is an induced path of $G$.
%This is possible by Lemma~\ref{lem:long}, since $n_{i-1}/d_{i-1}=n/d$ for every $i$.

\subsection{Stitching the paths}\label{s: stitch}
For every $i\in[k]$, denote by $P_i^-$ and $P_i^+$ the initial third and the final third of $P_i$, respectively. 
For $\square\in\{+,-\}$, denote by $W_i^\square$ the set of vertices in $V(G(i-1))\setminus V(P_i)$ which have exactly one neighbour in $P_i$, and whose unique neighbour in $P_i$ belongs to $P_i^\square$.

\begin{lemma}\label{lem:many-leaves}
With probability $1-O(n^{-3})$, for every $i\in[k]$ and every $\square\in\{+,-\}$,
\[
    |W_i^\square|\ge \frac{\eps^3}{1000}n_{i-1}.
\]
\end{lemma}

\begin{proof}
We condition on the event from Lemma~\ref{l: expansion} (which holds with the required probability).
Since $\theta_i\ll\eps$ and $|V(P_i)|=r$, by~\eqref{eq: path-length} the set $V(P_i)$ lies in the range from Lemma~\ref{l: expansion}. Thus, we obtain
\[
    |N_{G(i-1)}(P_i)|\ge (1-O(\eps))d_{i-1}|P_i|.
\]
Moreover, the total number of edges incident with $P_i$ in $G(i-1)$ is at most $(1+\delta_i)d_{i-1}|P_i|\le 1.1 d_{i-1} |P_i|$. 
Hence, all but $O(\eps d_{i-1}|P_i|)$ vertices in $N_{G(i-1)}(P_i)$ have exactly one neighbour in $P_i$; denote this set by $N_{G(i-1)}^u(P_i)$.
%We denote by $N_{G(i-1)}^u(P_i)$ the set of vertices in $N_{G(i-1)}(P_i)$ that have a unique neighbour in $P_i$. 
%In particular, the number of such leaves is at least $0.9d_{i-1}|P_i|$ for sufficiently large $n$.

Now, each of the subpaths $P_i^-$, $P_i^+$ and $P_i\setminus(P_i^-\cup P_i^+)$ is incident with at most $(1+\delta_i)d_{i-1}|P_i|/3 < 0.4d_{i-1}|P_i|$ vertices of $N_{G(i-1)}^u(P_i)$. Since $|N_{G(i-1)}^u(P_i)|\ge (1-O(\eps))d_{i-1}|P_i|$, each of $P_i^-$ and $P_i^+$ is incident with at least $d_{i-1}|P_i|/10$ vertices of $N_{G(i-1)}^u(P_i)$. Thus, the claim follows from \eqref{eq: path-length}.
\end{proof}

\begin{lemma}\label{lem:restrict}
With probability $1-O(n^{-2})$, for every $i\in[k]$, there are sets $Z_i^-\subseteq W_i^-$ and $Z_i^+\subseteq W_i^+$, each of size at least $(\lambda/d)^{0.1}n$, such that $Z_i^-\cup Z_i^+$ has no neighbours in $Y_{i+1}\cup\cdots\cup Y_k$.
\end{lemma}
\begin{proof}
We assume the event from Lemma~\ref{l: expansion} holds for $i\in [k]$ (this happens with probability $1-O(n^{-3})$) and construct $Z_i^+$; the construction of $Z_i^-$ is identical. For every integer $j\in [0,k]$, define
\[
    z_j=\frac{\eps^3}{2000}\e^{-2aj}n.
\]
By Lemma~\ref{lem:many-leaves}, $|W_i^+|\ge z_i$. Set $Z_{i,i}\subseteq W_i^+$ to be any subset of size $z_i$. We will recursively construct sets
\[
    Z_{i,i}\supseteq Z_{i+1,i}\supseteq\cdots\supseteq Z_{k,i}
\]
such that $|Z_{j,i}|\ge z_j$ and $Z_{j,i}$ has no neighbours in $Y_{i+1}\cup\cdots\cup Y_j$.

Suppose $Z_{j,i}$ has been constructed for some integer $j\in [i,k-1]$. Let
\[
    T_{j,i}=\{v\in Z_{j,i}:e_G(v,V(G(j)))>1.2d_j\}.
\]
By the property $\EM(n,(1\pm \gamma)d,\lambda)$, we have
\[
    1.2d_j|T_{j,i}|
    \le \frac{(1+\gamma)^2d}{(1-\gamma)n}|T_{j,i}||V(G(j))|+\frac{1+\gamma}{1-\gamma}\lambda\sqrt{|T_{j,i}||V(G(j))|}
    \le 1.1d_j|T_{j,i}|+\lambda\sqrt{2n_j|T_{j,i}|},
\]
and therefore
\[
    |T_{j,i}|\le 200\left(\frac{\lambda}{d_j}\right)^2n_j
    \le 200\left(\frac{\lambda}{d}\right)^2\e^{aj}n.
\]
Since $j\le k$ and $\e^{ak}\le (d/\lambda)^{2c_0}$ with $c_0\ll 1$, this is at most $\eps z_j$.

By definition, every vertex in $Z_{j,i}\setminus T_{j,i}$ has at most $1.2d_j$ neighbours in $G(j)$. 
Denote by $X$ the number of vertices of $Z_{j,i}\setminus T_{j,i}$ which have no neighbour in $Y_{j+1}$. Then, using $p_{j+1}=(1+\eps)/d_j$ and $\eps\ll 1$, we have
\[
    \mathbb E[X] \ge (1-\eps)z_j(1-p_{j+1})^{1.2d_j}\ge (1-2\eps)z_j\e^{-1.3}\ge 2z_{j+1},
\]
where the expectation is taken over the choice of $Y_{j+1}$.
Since $G$ is a $d$-regular graph, $X$ is a $d$-Lipschitz function of the random set $Y_{j+1}\subseteq V(G(j))$. Hence, by Lemma~\ref{azuma} and recalling that $d^6\le n$, that $|V(G(j))|=(1\pm 2\delta_j)n_j$ and the choice of $c_0$,
\begin{align*}
    \mathbb P(|X-\mathbb E[X]|\ge z_{j+1})&\le \exp\bigg(-\frac{z_{j+1}^2}{2d^2|V(G(j))|p_{j+1}+2dz_{j+1}/3}\bigg)\le\exp\bigg(-\frac{z_{j+1}^2}{2d^2\cdot 2n/d+2dz_{j+1}/3}\bigg)\\
    &\le \exp\bigg(-\frac{z_{j+1}^2}{5dn}\bigg)\le \exp\bigg(-\frac{\eps^6 \e^{-4ak} n}{5\cdot 2000^2 d}\bigg)\le \e^{-\sqrt{n}}.
\end{align*}
Taking $Z_{j+1,i}$ to be the set of any $z_{j+1}$ surviving vertices completes the induction.
A union bound over the $O(k^2)=o(n)$ pairs $(i,j)$ proves the assertion. Finally, by the choice of $c_0\ll c_1\ll 1$,
\[
    z_k=\frac{\eps^3}{2000}\e^{-2ak}n\ge \bigg(\frac{\lambda}{d}\bigg)^{0.1}n.\qedhere
\]
\end{proof}

\subsection{\texorpdfstring{Proof of \Cref{th: main} for $d^6\le n$}{Proof of Theorem~1.1 for sparse graphs}}\label{s: proof_d<<n}

We work on the intersection of the events from Lemmas~\ref{l: validity}, \ref{l: expansion}, \ref{l: excess}, \ref{lem:long} and~\ref{lem:restrict}, which has positive probability. Thus, all objects constructed above may be fixed deterministically.

Consider indices modulo $k$. To finish the construction, we find edges $e_1,\ldots,e_k$ at distance at least two from each other such that, for every $i\in [k]$, $e_i$ goes between $Z_i^+$ and $Z_{i+1}^-$.
Our construction is iterative, namely, suppose that, for some $i\in [1,k-1]$, we have already chosen edges $e_1,\ldots,e_i$ with the desired property.
Let
\[
    A_i=Z_{i+1}^+\setminus N_G[e_1\cup\cdots\cup e_i],
    \qquad
    B_i=Z_{i+2}^-\setminus N_G[e_1\cup\cdots\cup e_i].
\]
We note that the closed neighbourhood above has size at most $2i(d+1)\le 2k(d+1) = o((\lambda/d)^{0.1}n)$. 
%By Lemma~\ref{l: spectral-lower}, $\lambda/d\ge d^{-1/2}/2$ for large $n$, and by \eqref{eq: d-small},
%\[
%    3dt\le 3dk\ll (\lambda/d)^{0.1}n.
%\]
%Since the closed neighbourhood of the $2t$ endpoints chosen so far has size at most $3dt$, 
Recalling that the sets $Z_{i+1}^+,Z_{i+2}^-$ have size at least $(\lambda/d)^{0.1}n$, it follows that $|A_i|,|B_i|\ge (\lambda/d)^{0.1}n/2$. 
Therefore, the property $\EM(n,(1\pm \gamma)d,\lambda)$ gives $e(A_i,B_i)>0$ since
\[
    \frac{(1+\gamma)^2d}{(1-\gamma)n}|A_i||B_i|>\frac{1+\gamma}{1-\gamma}\lambda\sqrt{|A_i||B_i|}
\]
whenever $c_1$ is sufficiently small. Choose such an edge and call it $e_{i+1}$. This greedy procedure yields an induced matching of connecting edges $e_i\in E_G(Z_i^+,Z_{i+1}^-)$ for all $i\in[k]$.

We justify that this suffices to provide the desired long induced cycle.
For every $i\in [k]$, each endpoint of $e_i$ has a unique neighbour in the appropriate end segment of $P_i$ or $P_{i+1}$, and has no neighbours in any other path $P_j$: for earlier paths, this follows from the fact that $V(G(i))$ at distance more than 1 from $Y_1\cup\ldots\cup Y_{i-1}$, and for later paths, it follows from Lemma~\ref{lem:restrict}. 
The connecting edges themselves form an induced matching. 
Hence, the union over $i\in [k]$ of the subpath of $P_i$ between the neighbour of $e_{i-1}$ (contained in $P_i^-$) and the neighbour of $e_i$ (contained in $P_i^+$) with the edges $e_1,\ldots,e_k$ forms an induced cycle. 
Moreover, for each $i$, the used subpath of $P_i$ has length at least $|P_i|/3$. Therefore the length of the resulting induced cycle is at least
\[
    \frac{k r}{3}
    \ge \frac{c_0\eps^3}{400}\frac{n\log(d/\lambda)}{d}
\]
for all sufficiently large $n$. Taking $c_2\le c_0\eps^3/400$ completes the proof of \Cref{th: main} when $d^6\le n$.

\subsection{Extension to the full range of \texorpdfstring{$d=d(n)$}{d=d(n)}}\label{s: full d}

We have shown that every $(n, (1\pm\gamma)d, \lambda)$-graph with $\gamma=\left(\frac{\lambda}{d}\right)^{1/3}$, $d^6\le n$ and $\lambda\le c_1 d$, contains an induced cycle of length at least
\[
    c_2\frac{n\log(d/\lambda)}{d}.
\]
We aim to extend this result to the full range of $d=d(n)$. We use this result as the base case and proceeds iteratively, establishing the result for progressively larger values of $d$. The first key argument is as follows.

\begin{lemma}\label{lem: iteration_1}
Fix constants $s\in [1/6,1)$ and $c_2\ll c_1,c_3\ll 1$.
Assume that, for all large enough $d\le n^s$, 
every $\left(n,\left(1\pm c_3(\frac{\lambda}{d})^{1/3}\right)d,\lambda\right)$-graph with $\lambda\le c_1 d$ contains an induced cycle of length at least
    \[
        c_2\frac{n\log(d/\lambda)}{d}.
    \] 
Then, for all large enough $d\le n^{1/2+s/2}$, every $\left(n,\left(1\pm \frac{c_3}{2}(\frac{\lambda}{d})^{1/3}\right)d,\lambda\right)$-graph $G$ with $\lambda\le \frac{c_1}{11} d$ contains an induced cycle of length at least
    \[
        \frac{c_2}{100}\cdot \frac{n\log\left(d/\lambda\right)}{d}.
    \] 
\end{lemma}

Our proof of Lemma~\ref{lem: iteration_1} relies on the following theorem which is a parameter-optimised variant of an induced-subgraph inheritance theorem of Ferber, Han, Mao, and Vershynin \cite[Theorem 6.2]{FHMV25}. Roughly speaking, their result shows that random induced subgraphs of pseudorandom graphs inherit pseudorandomness with high probability.

\begin{theorem}\label{thm:patched} 
There is an absolute constant $C>0$ such that the following holds for all sufficiently large $n$.
Fix $d,\lambda>0$, $\gamma\in(0,1/200]$ and $\sigma\in [C/n,1)$ with $m=\lfloor \sigma n\rfloor \in\mathbb N$, and an $(n,(1\pm\gamma)d,\lambda)$-graph $G$. 
Consider a uniformly random subset $X$ of $[n]$ of size $m$, and let $H:=G[X]$. Assume
\begin{equation}\label{eq:patched-hypotheses}
        \sigma d\ge C\gamma^{-2}\log m
        \qquad \text{and}\qquad
        \sigma\lambda\ge C\sqrt{\sigma d\log m}.
\end{equation}
Then, with probability at least $1-m^{-1/100}$, $H$ is an $(m,(1\pm2\gamma)\sigma d,11\sigma\lambda)$-graph.
\end{theorem}

Note that the theorem becomes trivial once $\lambda$ is very close to $d$. However, we recall that, when applying the theorem, we will have in fact that $\lambda<c_1d$ where $c_1\ll 1$.

The proof of~\Cref{thm:patched} is quite technical and provided in~\Cref{s: generalisation of FHMV}. Our contribution in the proof of~\Cref{thm:patched} is to sharpen the quantitative form needed for our bootstrapping argument: using an improved operator-norm estimate for random submatrices, we obtain the same inheritance conclusion for a wider range of degrees and spectral parameters.

%We observe that if a graph $G$ is an $(n,(1\pm\gamma)d,\lambda)$-graph, then it is also an $(n,(1\pm\gamma)d,\lambda')$-graph for $\lambda'\ge\lambda$, hence we can apply \Cref{thm:patched} on $G$ with $\lambda'$ instead of $\lambda$. We use this observation in the proof of Lemma~\ref{lem: iteration_1}.

\begin{proof}[Proof of Lemma \ref{lem: iteration_1}]
First, we may (and do) assume that $d\in [n^s,n^{1/2+s/2}]$, as the case where $d\le n^s$ is covered by the assumption of the lemma.
Fix a graph $G$ as described in the statement, set 
\[\lambda' = \max\{\lambda, n^{-1/50}d\},\qquad p\coloneqq \frac{n^{s/2}}{d}\qquad\text{and}\qquad \gamma=\frac{c_3}{2}\left(\frac{\lambda'}{d}\right)^{1/3},\] 
and note that $G$ is an $\left(n,\left(1\pm \gamma\right)d,\lambda'\right)$-graph such that $pd\ge C\gamma^{-2}\log (np)$ and $p\lambda'\ge C\sqrt{pd\log (np)}$, where the constant $C>0$ is the one from Theorem~\ref{thm:patched}.
Consider a uniformly chosen set $X\subseteq V(G)$ of size $|X|=pn$. 
Then, whp $G[X]$ is a $\left(pn, (1\pm 2\gamma)pd, 11p\lambda' \right)$-graph where, using that $d\le n^{1/2+s/2}$ and $\lambda'\le c_1d/11$ for large $n$,
\[pd \le (pn)^{s},\qquad 11p\lambda'\le c_1 pd\qquad \text{and}\qquad 2\gamma=c_3\left(\frac{\lambda'}{d}\right)^{1/3}\le c_3\left(\frac{11\lambda' p}{dp}\right)^{1/3}.\]
Hence, using that $d/\lambda'\ge n^{1/50}\ge (d/\lambda)^{1/50}$ when $\lambda'>\lambda$ and choosing $c_1$ small constant, $G[X]$ (and therefore $G$ as well) contains an induced cycle of length at least
    \begin{align*}
        c_2\frac{pn\log(pd/11p\lambda')}{pd}&=c_2\frac{n\log(d/11\lambda')}{d}\ge \frac{c_2}{2}\cdot \frac{n\log(d/\lambda')}{d}\ge \frac{c_2}{100}\cdot \frac{n\log(d/\lambda)}{d}.\qedhere
    \end{align*}
\end{proof}

The following corollary is an iterated version of the previous lemma.
\begin{corollary}\label{lem: concluding iterations}
Fix constants $c_2\ll c_1,c_3\ll 1$. 
For every integer $t\ge 0$ and large enough integers $d,n$ with $d\le n^{1-5/(6\cdot 2^t)}$, every $\left(n,\left(1\pm \frac{c_3}{2^t}(\frac{\lambda}{d})^{1/3}\right)d,\lambda\right)$-graph $G$ with $\lambda\le c_1 d/11^t$ contains an induced cycle of length at least
    \[
        \frac{c_2}{100^t}\cdot \frac{n\log(d/\lambda)}{d}.
    \]
%, the following holds. Let $n,d$ be sufficiently large integers such that $d\le n^{1-\frac{9}{10}\cdot \left(\frac{1}{2}\right)^t}$. Then, every $\left(n,\left(d,\lambda\right)$-graph $G$ with $\lambda\le \frac{c_1}{6^t} d$ contains an induced cycle of length at least
%    
\end{corollary}
\begin{proof}
We prove the statement by induction on $t$. The case $t=0$ follows from the results proved in Sections~\ref{s: setup}-\ref{s: proof_d<<n}.
Assume the statement holds with $t-1$ for some $t\ge 1$. 
By the induction hypothesis, whenever $d\le n^{1-5/(6\cdot 2^{t-1})}$, the result of the lemma holds. Hence, by applying Lemma~\ref{lem: iteration_1} with 
$s \coloneq 1-5/(6\cdot 2^{t-1})$ and $c_1/11^{t-1}, c_2/100^{t-1}, c_3/2^{t-1}$ replacing $c_1,c_2,c_3$, respectively, whenever $d\le n^s$ and $\lambda\le c_1d/11^t$, we have that $G$ contains an induced cycle of length at least
    \[
        \frac{c_2}{100^t}\cdot \frac{n\log(d/\lambda)}{d}.
    \]
It remains to observe that $1/2+s/2=1-5/(6\cdot 2^t)$, which implies the corollary.
\end{proof}

The following lemma extends the construction of a long induced cycle to the full range of $d=d(n)$, for small values of $\lambda$ in terms of $n,d$. Its proof uses a very different approach and is based on deterministically constructing an induced path in $G$ vertex-by-vertex and finally closing it into an induced cycle.

\begin{lemma}\label{lem: one by one}
For all $\alpha\in (0,1)$, there are constants $c_2\ll c_1,c_3\ll \alpha$ such that, for all large enough $d,n$ with $d\le \alpha n$, every $\left(n,\left(1\pm c_3(\frac{\lambda}{d})^{1/3}\right) d,\lambda\right)$-graph $G$ with $\lambda\le c_1 \left(\frac{d}{n}\right)^{5/4}\cdot d$ contains an induced cycle of length at least
    \[
        c_2\frac{n\log(d/\lambda)}{d}.
    \]
\end{lemma}

\begin{proof}
Fix $\delta\coloneqq \delta(\alpha)>0$ such that $(1+2\delta)\alpha<1-\delta$ and $(5c_3c_1^{1/3}+\alpha)(1+\delta)<1$; note that such $\delta$ exists since $c_1,c_3\ll \alpha < 1$ and satisfies $1-(1+2\delta)d/n>\delta$.
%Notice that such a $\delta$ exists since $d\le \alpha n$ with $0<\alpha<1$ and since $c_1, c_3$ is sufficiently small in terms of $\alpha$. Notice that this $\delta$ satisfies $1-(1+2\delta)\frac{d}{n}>\delta$. 
Also, for $\gamma=c_3\left(\frac{\lambda}{d}\right)^{1/3}\ll \delta$, we have that
\begin{equation}\label{eq:bds}
(1+\delta)\frac{(1+\gamma)^2}{1-\gamma} \le 1+\delta+4\gamma < 1+2\delta\qquad \text{and}\qquad (1-\delta)\frac{(1-\gamma)^2}{1+\gamma} \ge 1-\delta-4\gamma > 1-2\delta.
\end{equation}
We will make use of these inequalities later in the proof.
Next, we assume that $d>n^{3/4}$ (as Corollary~\ref{lem: concluding iterations} shows the result for smaller values of $d$). Also, fix $v_0 \in V(G)$ 
%Observe that we can assume that $d>n^{3/4}$ since Lemma~\ref{lem: small d} shows the result of \Cref{th: main} for smaller values of $d$. 
and define the following sets:
\begin{align*}
    V_0=V(G),\qquad  A_0=V_0\cap N(v_0),\qquad  B_0=V_0\setminus A_0 \qquad  \text{and} \qquad  U_0=N(v_0).
\end{align*}
Observe that 
\begin{equation}\label{eq: sizes_0}
\begin{split}
    &|V_0|=n \left(1-\left(1\pm 2\delta\right)\frac{d}{n}\right)^0,\qquad  |A_0|=e(v_0,V_0) =\left(1\pm 2\delta\right)\frac{d}{n}|V_0|,\\
    &|B_0|=\left(1-\left(1\pm 2\delta\right)\frac{d}{n}\right)|V_0| \qquad  \text{and} \qquad  |U_0|\ge (1-\delta)d\left(1-\left(1+ 2\delta\right)\frac{d}{n}\right)^0.
\end{split}
\end{equation}
We show that there is a vertex $v_1\in A_0\setminus\{v_0\}$ with 
\begin{equation}\label{eq:propsv1}
e(v_1,U_0)<\left(1+ 2\delta\right)\tfrac{d}{n}|U_0|-1 \qquad \text{and}\qquad e(v_1,B_0)\in [(1- 2\delta)\tfrac{d}{n}|B_0|,(1+2\delta)\tfrac{d}{n}|B_0|]. 
\end{equation}
Set 
\begin{align*}
A^+_0&=\big\{v\in A_0\mid e(v,U_0)>(1+\delta)\tfrac{(1+\gamma)^2}{1-\gamma}\tfrac{d}{n}|U_0|\big\},\\
B_0^+ &= \big\{v\in A_0\setminus A_0^+\mid e(v,B_0)>(1+\delta)\tfrac{(1+\gamma)^2}{1-\gamma}\tfrac{d}{n}|B_0|\big\},\\
B_0^- &= \big\{v\in A_0\setminus A_0^+\mid e(v,B_0)<(1-\delta)\tfrac{(1-\gamma)^2}{1+\gamma}\tfrac{d}{n}|B_0|\big\}.  
\end{align*}

\begin{claim}\label{cl:A0}
We have $|A^+_0|<\delta|A_0|$ and $|B^+_0\cup B^-_0| < |A_0\setminus A_0^+|/2$. 
\end{claim}
\begin{proof}
First, we deal with the first inequality.
By Lemma~\ref{cor:EML-almost-regular},
\[
    e(A^+_0,U_0)\le \frac{(1+\gamma)^2}{1-\gamma}\cdot\frac{d}{n}|A^+_0||U_0|+\frac{1+\gamma}{1-\gamma}\cdot\lambda \sqrt{|A^+_0||U_0|}.
\]
However, by the definition of $A^+_0$, 
\[
    e(A^+_0,U_0)> \left(1+\delta\right)\frac{(1+\gamma)^2}{1-\gamma}\cdot\frac{d}{n}|A^+_0||U_0|,
\]
and combining the two inequalities implies
\[\frac{1+\gamma}{1-\gamma}\cdot\lambda \sqrt{|A^+_0||U_0|}\ge \delta \frac{(1+\gamma)^2}{1-\gamma}\cdot\frac{d}{n}|A^+_0||U_0|\quad\implies\quad |A^+_0|\le \bigg(\frac{\lambda n}{\delta (1+\gamma)d}\bigg)^2 \frac{1}{|U_0|}\le \bigg(\frac{d}{n}\bigg)^{1/2}\frac{\sqrt{c_1} d^2}{|U_0|}\le \delta |A_0|,\]
%\[
%    \lambda >\delta(1+\gamma)\frac{d}{n}\sqrt{|A'||U_0|}.
%\]
%The assumption $|A'|\ge\delta|A_0|$ together with \eqref{eq: sizes_0} yield the following:
%\begin{align*}
%    \lambda &>\delta(1+\gamma)\frac{d}{n}\sqrt{\delta\left(1- 2\delta\right)\frac{d}{n}\cdot n\cdot d}\ge \delta^{1.5}\sqrt{1- 2\delta}\cdot \frac{d^2}{n},
%\end{align*}
where we used that $c_1\ll \alpha$ and that $\delta=\delta(\alpha)$.
Next, we separately show that $|B^+_0| < |A_0\setminus A^+_0|/4$ and $|B^-_0| < |A_0\setminus A^+_0|/4$, starting with the first inequality.
By Lemma~\ref{cor:EML-almost-regular},
\[
    e(B^+_0,B_0)\le \frac{(1+\gamma)^2}{1-\gamma}\cdot\frac{d}{n}|B^+_0||B_0|+\frac{1+\gamma}{1-\gamma}\cdot\lambda \sqrt{|B^+_0||B_0|},
\]
and by the definition of $B^+_0$,
\[
    e(B^+_0,B_0)> \left(1+\delta\right)\frac{(1+\gamma)^2}{1-\gamma}\cdot\frac{d}{n}|B^+_0||B_0|.
\]
The last two displays together imply that
\[\frac{1+\gamma}{1-\gamma}\cdot\lambda \sqrt{|B^+_0||B_0|}\ge \delta \frac{(1+\gamma)^2}{1-\gamma}\cdot\frac{d}{n}|B^+_0||B_0|\quad\implies\quad |B^+_0|\le \bigg(\frac{\lambda n}{\delta (1+\gamma)d}\bigg)^2 \frac{1}{|B_0|}\le \bigg(\frac{d}{n}\bigg)^{1/2}\frac{\sqrt{c_1} d^2}{|B_0|} < \frac{|A_0\setminus A_0^+|}{4}.\]

Proving that $|B^-_0| < |A_0\setminus A^+_0|/4$ follows an analogous computation. Together with the last display, it implies that $|B^+_0\cup B^-_0| < |A_0\setminus A_0^+|$ and finishes the proof.
\end{proof}

By Claim~\ref{cl:A0},~\eqref{eq:bds} and~\eqref{eq: sizes_0}, there exists $v_1\in A_0\setminus\{v_0\}$ satisfying~\eqref{eq:propsv1}. 
The induced path we build starts with the edge $v_0v_1$. 

Next, define $K$ to be the largest integer satisfying $(1-(1+2\delta)d/n)^K > \lambda n/d^2$. In particular, we obtain that 
\[K \in \bigg[\frac{n\log(d^2/\lambda n)}{2d},\frac{2n\log(d^2/\lambda n)}{d}\bigg]\quad\text{and also}\quad K\ge \frac{n}{2d} \log\bigg(\frac{d}{\lambda}\cdot \frac{d}{n}\bigg)\ge \frac{n}{2d} \log\bigg(\frac{d}{\lambda}\cdot \bigg(\frac{\lambda}{d}\bigg)^{4/5}\bigg)= \frac{n}{10d} \log\bigg(\frac{d}{\lambda}\bigg),\]
where we used that $\lambda\le (d/n)^{5/4} d$.
Next, for a positive integer $i < K/2$, suppose that we found an induced path $v_0,\ldots,v_i$ and define the sets

%\begin{align*}
%    V_1=B_0=V\setminus N(v_0) ,\quad  A_1=V_1\cap N(v_1),\quad  B_1=V_1\setminus A_1, \quad  \text{and} \quad  U_1=N(v_0)\setminus N(v_1).
%\end{align*}
%Observe that 
%\begin{align}\label{eq: sizes_1}\notag
%    &|V_1|=n\cdot \left(1-\left(1\pm 2\delta\right)\frac{d}{n}\right)^1,\quad  |A_1|=\text{deg}(v_1,V_1)=\left(1\pm 2\delta\right)\frac{d}{n}|V_1|,\\
%    &|B_1|=\left(1-\left(1\pm 2\delta\right)\frac{d}{n}\right)|V_1|, \quad  \text{and} \quad  |U_1|\ge d\left(1-\left(1+ 2\delta\right)\frac{d}{n}\right)^1.
%\end{align}

%Let $k\in \mathbb{N}$ be the last integer satisfying $\left(1-(1+2\delta)\frac{d}{n}\right)^k>\frac{\lambda n}{d^2}$. Observe that $k=\Theta\left(\frac{n\log\left(\frac{d^2}{\lambda n}\right)}{d}\right)$. Notice that since $\lambda\le c_1 \left(\frac{d}{n}\right)^{5/4}\cdot d$, we have $k=\Omega\left(\frac{n\log\left(d/\lambda\right)}{d}\right)$. We show by induction that the construction of an induced path in $G$ continues as long as its length reaches $k/2$. Assume we have an induced path $v_0 v_1\ldots v_i$ in $G$ with $i<k/2$ and sets 
\begin{align*}
    &V_i=B_{i-1}=V\setminus N[\{v_0,\ldots,v_{i-1}\}] ,\quad  A_i=V_i\cap N(v_i),\\
    &B_i=V_i\setminus A_i \quad  \text{and} \quad  U_i=U_{i-1}\setminus N[v_i]=N(v_0)\setminus N[\{v_1,\ldots,v_{i}\}].
\end{align*}
We also assume the properties 
\begin{equation}\label{eq: sizes_i}
\begin{split}
    &|V_i|=n \left(1-\left(1\pm 2\delta\right)\frac{d}{n}\right)^i,\quad  |A_i|=e(v_i,V_i)=\left(1\pm 2\delta\right)\frac{d}{n}|V_i|,\\
    &|B_i|=\left(1-\left(1\pm 2\delta\right)\frac{d}{n}\right)|V_i| \quad  \text{and} \quad  |U_i|\ge (1-\delta)d\left(1-\left(1+2\delta\right)\frac{d}{n}\right)^i;
\end{split}
\end{equation}
note that they are verified for $i=1$ by combining~\eqref{eq: sizes_0} and~\eqref{eq:propsv1}. Assuming them for $i$, we show that there exists a vertex $v_{i+1}\in A_i\setminus\{v_i\}$ such that 
\begin{equation}\label{eq:propsvi}
e(v_{i+1},U_i)<\left(1+ 2\delta\right)\tfrac{d}{n}|U_i|-1\qquad \text{and}\qquad e(v_{i+1},B_i)\in [(1- 2\delta)\tfrac{d}{n}|B_i|,(1+2\delta)\tfrac{d}{n}|B_i|]. 
\end{equation}
Set
\begin{align*}
A_i^+ 
&= \big\{v \in A_i \mid e(v,U_i) > (1+\delta)\tfrac{(1+\gamma)^2}{1-\gamma}\tfrac{d}{n}|U_i|\big\},\\[6pt]
B_i^+ 
&= \big\{v \in A_i\setminus A_i^+ \mid e(v,B_i) > (1+\delta)\tfrac{(1+\gamma)^2}{1-\gamma}\tfrac{d}{n}|B_i|\big\},\\[6pt]
B_i^- 
&= \big\{v \in A_i\setminus A_i^+ \mid e(v,B_i) < (1-\delta)\tfrac{(1-\gamma)^2}{1+\gamma}\tfrac{d}{n}|B_i|\big\}.
\end{align*}

\begin{claim}\label{cl:Ai}
We have $|A_i^+| < \delta |A_i|$ and $|B_i^+ \cup B_i^-| < |A_i \setminus A_i^+|/2$.
\end{claim}

The proof is completely analogous to the one of Claim~\ref{cl:A0} but provided for completeness.

\begin{proof}
We first prove that $|A_i^+| < \delta |A_i|$. By Lemma~\ref{cor:EML-almost-regular},
\[
    e(A_i^+,U_i)
    \le 
    \frac{(1+\gamma)^2}{1-\gamma}\cdot\frac{d}{n}|A_i^+||U_i|
    +\frac{1+\gamma}{1-\gamma}\cdot \lambda \sqrt{|A_i^+||U_i|}.
\]
On the other hand, by the definition of $A_i^+$,
\[
    e(A_i^+,U_i)
    >
    (1+\delta)\frac{(1+\gamma)^2}{1-\gamma}\cdot\frac{d}{n}|A_i^+||U_i|,
\]
and combining the two inequalities implies
\begin{equation}\label{eq:displi1}
\frac{1+\gamma}{1-\gamma}\cdot \lambda \sqrt{|A_i^+||U_i|}
\ge 
\delta \frac{(1+\gamma)^2}{1-\gamma}\cdot\frac{d}{n}|A_i^+||U_i|\quad\implies\quad
|A_i^+|
\le
\bigg(\frac{\lambda n}{\delta (1+\gamma)d}\bigg)^2 \frac{1}{|U_i|}.
\end{equation}
Using that $2i\le K$, we deduce that
\begin{align}\label{eq:displi2}\notag
\bigg(\frac{\lambda n}{\delta (1+\gamma)d}\bigg)^2 \frac{1}{|U_i|\cdot |A_i|}\le \bigg(\frac{\lambda n}{\delta (1+\gamma)d}\bigg)^2 \frac{1}{(1-2\delta)^2 d^2(1-(1+2\delta)d/n)^{2i}}\\
\le \bigg(\frac{\lambda n}{\delta (1+\gamma)d}\bigg)^2\frac{1}{(1-2\delta)^2\lambda n}\le\frac{1}{\delta^2(1+\gamma)^2(1-2\delta)^2}\cdot \bigg(\frac{\lambda}{d}\bigg)^{1/5}\le \delta,
\end{align}
where we used that $\lambda/d\le (d/n)^{5/4}\ll c_1\ll \alpha$ and that $\delta=\delta(\alpha)$.
Combining the last two displays shows that $|A_i^+|\le \delta |A_i|$.

We turn to showing that $|B_i^+| < |A_i \setminus A_i^+|/4$; as before, $|B_i^-| < |A_i \setminus A_i^+|/4$ follows similarly.
By Lemma~\ref{cor:EML-almost-regular},
\[
    e(B_i^+,B_i)
    \le 
    \frac{(1+\gamma)^2}{1-\gamma}\cdot\frac{d}{n}|B_i^+||B_i|
    +\frac{1+\gamma}{1-\gamma}\cdot \lambda \sqrt{|B_i^+||B_i|}.
\]
Moreover, by the definition of $B_i^+$,
\[
    e(B_i^+,B_i)
    >
    (1+\delta)\frac{(1+\gamma)^2}{1-\gamma}\cdot\frac{d}{n}|B_i^+||B_i|,
\]
and combining these inequalities gives
\[
\frac{1+\gamma}{1-\gamma}\cdot \lambda \sqrt{|B_i^+||B_i|}
\ge 
\delta \frac{(1+\gamma)^2}{1-\gamma}\cdot\frac{d}{n}|B_i^+||B_i|,
\quad \implies\quad
|B_i^+|
\le 
\bigg(\frac{\lambda n}{\delta (1+\gamma)d}\bigg)^2 \frac{1}{|B_i|}.
\]
%Observe that by \eqref{eq: sizes_i}, we have $|B_i|\ge |U_i|$, therefore, combining it with~\eqref{eq:displi1} and~\eqref{eq:displi2}, we obtain $|B_i^+|\le \delta |A_i| < |A\setminus A_i^+|/4$. 
As a result, using \eqref{eq: sizes_i}, the inequality $\lambda/d\le c_1\ll \delta(\alpha)$ and the fact that $i<K/2$ gives
\begin{align*}
|B_i^+|\le \bigg(\frac{\lambda n}{\delta (1+\gamma)d}\bigg)^2 \frac{\delta |A_i|}{\delta |A_i|\cdot |B_i|}
&\le \bigg(\frac{\lambda n}{\delta (1+\gamma)d}\bigg)^2 \frac{\delta |A_i|}{\delta (1 + 2\delta) d (1-(1 - 2\delta)d/n) |V_i|^2/n}\\
&\le \bigg(\frac{\lambda n}{\delta (1+\gamma)d}\bigg)^2 \frac{\delta |A_i|}{\delta (1 + 2\delta) d (1-(1 - 2\delta)d/n)^{2i+1} n}\\
&\le \frac{1}{\delta^3 (1+\gamma)^2 (1+2\delta)}\cdot \frac{\lambda}{d}\cdot \delta |A_i|\le \delta |A_i|.
\end{align*}
Using the analogous inequality for $B^-_i$ shows that $|B_i^+ \cup B_i^-| < |A_i \setminus A_i^+|/2$, which completes the proof.
\end{proof}

By Claim~\ref{cl:Ai},~\eqref{eq:bds} and~\eqref{eq: sizes_i},
%the relation $\min\{|U_i|,|B_i|\}\ge (1-2\delta)d\sqrt{\lambda n/d^2}\ge \sqrt{n}$,
there exists $v_{i+1}\in A_i\setminus\{v_i\}$ satisfying~\eqref{eq:propsvi}. 
Furthermore, it is easy to verify that this vertex ensures the properties in~\eqref{eq: sizes_i} for $i+1$.

Notice that, by construction, the path $v_0v_1\ldots v_{i+1}$ is an induced path in $G$ for every integer $i < \lfloor K/2\rfloor$. 
Finally, when $i+1=\lfloor K/2\rfloor$, say, we show that $e(v_{i+1},U_i)>0$, that is, we show that we can choose $v_{i+1}$ such that it has a neighbour $u$ in $U_i$.
Then, choosing a neighbour $u$ of vertex $v_{i+1}$ in the set $U_i$ suffices to close an induced cycle $v_0,v_1,\ldots,v_{i+1},u$: note that the vertex $u$ is not adjacent to any of $v_1,\ldots,v_i$ by definition of $U_i$.

It remains to show that $e(v_{i+1},U_i)>0$ for our choice of $i$. To this end, set 
\[A_i^- = \big\{v \in A_i \mid e(v,U_i) < (1-\delta)\tfrac{(1-\gamma)^2}{1+\gamma}\tfrac{d}{n}|U_i|\big\}.\]
By Lemma~\ref{cor:EML-almost-regular}, we have
\[
    e(A_i^-,U_i)
    \ge 
    \frac{(1-\gamma)^2}{1+\gamma}\cdot\frac{d}{n}|A_i^-||U_i|
    -\frac{1-\gamma}{1+\gamma}\cdot \lambda \sqrt{|A_i^-||U_i|}.
\]
At the same time, by the definition of $A_i^-$,
\[
    e(A_i^-,U_i)
    <
    (1-\delta)\frac{(1+\gamma)^2}{1-\gamma}\cdot\frac{d}{n}|A_i^-||U_i|,
\]
and combining the two inequalities implies
\begin{equation*}
\frac{1-\gamma}{1+\gamma}\cdot \lambda \sqrt{|A_i^-||U_i|}
\ge 
\delta \frac{(1-\gamma)^2}{1+\gamma}\cdot\frac{d}{n}|A_i^-||U_i|\quad\implies\quad
|A_i^-|
\le
\bigg(\frac{\lambda n}{\delta (1-\gamma)d}\bigg)^2 \frac{1}{|U_i|} < |A_i|,
\end{equation*}
where the last inequality holds by~\eqref{eq:displi1} and~\eqref{eq:displi2}. This shows that some neighbour of $v_{i+1}$ belongs to $U_i$ and concludes the proof.
\end{proof}

We are now ready to complete the proof of \Cref{th: main}.

\begin{proof}[Proof of \Cref{th: main}]
We can (and do) assume that $d>n^{4/5}$ as Corollary~\ref{lem: concluding iterations} shows the result for smaller values of $d$. 
Fix $\hat c_3=c_3/16$ with $c_3$ as in Corollary~\ref{lem: concluding iterations}, $C$ as in \Cref{thm:patched} and set $\gamma=\hat c_3(\lambda/d)^{1/3}$.
We also assume that 
\begin{enumerate}[(a)]
    \item $d=o(n)$, as Lemma~\ref{lem: one by one} covers the case $d=\Theta(n)$, and
    \item $\lambda=\Omega((d/n)^{4/3}\cdot d)$, as Lemma~\ref{lem: one by one} treats the complementary case.
\end{enumerate}
    
Set $p=C n^3/d^4$. We verify that the conditions of \Cref{thm:patched} are satisfied. First, since $n^{4/5}<d\le n$, we have that $C/n\le p< 1$. Moreover, $\gamma\le 1/200$ and noting that $d=o(n)$,
assumption (b) gives
\[dp= C \frac{n^3}{d^3}= \Omega\bigg(\frac{C}{\hat c_3^2} \bigg(\bigg(\frac{n}{d}\bigg)^{4/3}\bigg)^{2/3} \log\bigg(C \frac{n^4}{d^4}\bigg)\bigg)= \Omega\bigg(\frac{C}{\hat c_3^2} \bigg(\bigg(\frac{d}{\lambda}\bigg)^{1/3}\bigg)^{2} \log\bigg(C \frac{n^4}{d^4}\bigg)\bigg)= \Omega(C\gamma^{-2}\log(np)).\]
Finally, since $\lambda=\Omega((d/n)^{4/3}\cdot d)$ and $d=o(n)$, we also have $\lambda p\ge C\sqrt{dp\log (np)}$. 
Therefore, by \Cref{thm:patched}, there is $X\subseteq V(G)$ of size $|X|=np$ such that $G[X]$ is an $(np,(1\pm 2\gamma)dp, 11\lambda p)$-graph. Now, observe that $dp\le (np)^{4/5}$ and we also have that $2\gamma<(c_3/8)\cdot(11\lambda p/dp)^{1/3}$ for small enough $c_3$. 
Thus, when $11\lambda p\le c_1 (dp)/1331$, applying Corollary~\ref{lem: concluding iterations} (for $t=3$) to the graph $G[X]$ ensures an induced cycle of length at least
    \[
        \frac{c_2}{10^6}\cdot\frac{np\log(dp/11\lambda p)}{dp}\ge \frac{c_2}{2\cdot 10^6}\cdot \frac{n\log(d/\lambda)}{d}.
    \]
Since every induced cycle in $G[X]$ is also induced in $G$, the proof is completed.
\end{proof}

\subsection{\texorpdfstring{Random induced subgraphs of $(n,(1\pm \gamma)d,\lambda)$-graphs}{Random Induced Subgraphs of approximately regular pseudorandom graphs}}\label{s: generalisation of FHMV}

This subsection is devoted to proving the key theorem (\Cref{thm:patched}) we used in~\Cref{s: full d}. We provide here the full statement again:

\begin{theorem}\label{thm:patched1} 
There is an absolute constant $C>0$ such that the following holds for all sufficiently large $n$.
Fix $d,\lambda>0$, $\gamma\in(0,1/200]$ and $\sigma\in [C/n,1)$ with $m=\lfloor \sigma n\rfloor \in\mathbb N$, and an $(n,(1\pm\gamma)d,\lambda)$-graph $G$. 
Consider a uniformly random subset $X$ of $[n]$ of size $m$, and let $H:=G[X]$. Assume
\begin{equation}\label{eq:patched-hypotheses1}
        \sigma d\ge C\gamma^{-2}\log m
        \qquad \text{and}\qquad
        \sigma\lambda\ge C\sqrt{\sigma d\log m}.
\end{equation}
Then, with probability at least $1-m^{-1/100}$, $H$ is an $(m,(1\pm2\gamma)\sigma d,11\sigma\lambda)$-graph.
\end{theorem}

Before proving \Cref{thm:patched1}, we introduce the necessary background and present several auxiliary lemmas.

%As explained in Section~\ref{s: outline}, we first prove \Cref{th: main} in the case where $d^6\le n$, and then bootstrap this result for larger values of $d$. The transition from small to large $d$ uses the following parameter-optimised variant of an induced-subgraph inheritance theorem of Ferber, Han, Mao, and Vershynin \cite[Theorem 6.2]{FHMV25}. Roughly speaking, their result shows that random induced subgraphs of psuedorandom graphs inherit pseudorandomness with high probability. The contribution here is to sharpen the quantitative form needed for our bootstrapping argument: using an improved operator-norm estimate for random submatrices, we obtain the same inheritance conclusion for a wider range of degrees and spectral parameters.

%The rest of this section is dedicated to the proof of Theorem~\ref{thm:patched1}.

\subsubsection{Matrix norms and their properties} We start with several classic properties of matrix norms.
Fix a real-valued matrix $A$ with dimensions $m\times n$.
The $\ell^2\to \ell^2$ operator norm of $A$ is defined as 
\[\|A\| = \max_{x \in \mathbb R^n\setminus \{0\}} \frac{\|Ax\|_2}{\|x\|_2} = \max_{x\in \mathbb S^{n-1}} \|Ax\|_2,\]
where $\|\cdot\|_2$ is the standard $\ell^2$-norm, i.e.,
\[
    \|x\|_2=x^T x=\sqrt{\sum_{i=1}^n x_i^2}, \quad \text{for every} \quad x = \begin{bmatrix}
    x_1 \\
    x_2 \\
    \vdots \\
    x_n
    \end{bmatrix}\in \mathbb{R}^n.
\]
We further recall the definition of singular values of a matrix. Given a real $m\times n$ matrix $A$, the singular values of $A$ are the
non-negative square roots of the eigenvalues of the symmetric positive semidefinite matrix $A^T A$. We denote by $s_i(A)$ the $i$-th singular value of $A$ in non-increasing order.

The following lemma contains several simple facts for the $\ell^2\to\ell^2$ operator norm and the spectra of matrices.

\begin{lemma}\label{lem:operator}
Each of the following holds:
\begin{enumerate}
    \item[\emph{(i)}] For every matrix $A$, we have $\|A\|=\|A^T\|$.
    \item[\emph{(ii)}] For every matrix $A$, every subset of rows $R$ of $A$ and every subset of columns $C$ of $A$, $\|A_{R\times C}\|\le \|A\|$. 
    \item [\emph{(iii)}] For every $n\times n$ block-diagonal symmetric matrix $A$ with non-negative entries, we have $|\lambda_i|\le \lambda_1$ for every $2\le i\le n$.
    \item[\emph{(iv)}] For every symmetric matrix $A$ with eigenvalues $\lambda_1\ge \ldots\ge \lambda_n$ such that $\lambda_1 \ge |\lambda_n|$, the second singular value $s_2(A)$ is equal to $\max\{|\lambda_2|,|\lambda_n|\}$.
\end{enumerate}
\end{lemma}
\begin{proof}
The first is a consequence of \cite[Chapter 17, page 17-3, fact 9(a)]{Hog06}, and the second point is a consequence of \cite[Chapter 17, page 17-7, fact 3]{Hog06}. The third point is a consequence of the Perron-Frobenius Theorem (see, for example, \cite[Chapter 9, page 9-4, fact 5]{Hog06}).
The last point follows from the fact that, by definition of the singular values of a diagonalisable real-valued matrix, $s_1,\ldots,s_n$ corresponds to a decreasing ordering of $|\lambda_1|,\ldots,|\lambda_n|$.
Since $\lambda_1\ge |\lambda_n|$, we have $s_1(A)=\lambda_1$ and therefore $s_2(A)=\max_{i\in [2,n]} |\lambda_i| = \max\{|\lambda_2|,|\lambda_n|\}$.
\end{proof}

We will also make use of the following lemma, which follows from Eckart-Young-Mirsky theorem; it provides another way to compute the second singular value of a matrix.

\begin{lemma}\label{lem:low-rank-approximation}
Let $A$ be an $m \times n$
matrix. Then
\[
s_2(A)=\min_B\:\|{A-B}\|,
\]
where the minimum is over all rank-one $m\times n$ matrices $B$. Moreover, the minimum is attained by $B = s_1(A)\textbf{u}_1\textbf{v}^T_1$, where 
$\textbf{v}_1\in \mathbb{R}^n$ and $\textbf{u}_1\in \mathbb{R}^m$ are any unit vectors such that $A\textbf{v}_1=s_1(A)\textbf{u}_1$.
\end{lemma}

The following lemma appears (and are proved) in \cite{FHMV25} in \cite{FHMV25}. We give here the proof for completeness.

\begin{lemma} \label{obs: |N|}
    Let $A$ be an $m \times n$ matrix with non-negative entries without zero rows or columns.
    Let $a \coloneqq \sum_{i,j}A_{ij}$ and let $\mathbbm{1}_n$ denote the vector in $\mathbb{R}^n$ whose all coordinates are equal to $1$. Define the following two diagonal matrices: Let $L=L(A)$ be the $m\times m$ diagonal matrix with $L_{i,i}=\sum_{j}A_{i,j}$ for all $i$ (that is, the sum of entries in the $i$-th row), and $R=R(A)$ be the $n\times n$ diagonal matrix with $R_{j,j}=\sum_{i}A_{i,j}$ (that is, the sum of entries in the $j$-th column).  Consider the vectors $\textbf{u}_1 \coloneqq a^{-1/2} L^{1/2} \mathbbm{1}_m$ and $\textbf{v}_1 \coloneqq a^{-1/2} R^{1/2} \mathbbm{1}_n$. Then: 
    \begin{enumerate}
        \item both $\textbf{u}_1$ and $\textbf{v}_1$ are unit vectors;
        \item $L^{-1/2}AR^{-1/2} \textbf{v}_1=\textbf{u}_1$;
        \item $s_1(L^{-1/2}AR^{-1/2}) = \|{L^{-1/2}AR^{-1/2}}\| = \textbf{u}_1^T L^{-1/2}AR^{-1/2} \textbf{v}_1 = 1$.
    \end{enumerate}
\end{lemma}

\begin{proof} 
The first two parts readily follow from the definitions of $a$, $L$, $R$.
As for the third part, the equation $s_1(L^{-1/2}AR^{-1/2}) = \|{L^{-1/2}AR^{-1/2}}\|$ holds for any matrix. Let us show that $\|L^{-1/2}AR^{-1/2}\|\le 1$. For every $\textbf{x}=(x_i)_{i=1}^n,\textbf{y}=(y_i)_{i=1}^n\in \mathbb{R}^n$ such that $\|\textbf{x}\|_2=\|\textbf{y}\|_2=1$, we have
\[0\le \sum_{i\in [m], j \in[n]}A_{i,j}\left(\frac{x_i}{\sqrt{L_{i,i}}}-\frac{y_j}{\sqrt{R_{j,j}}}\right)^2=2-2\sum_{i\in [m], j \in[n]}\frac{A_{i,j}x_iy_j}{\sqrt{L_{i,i}R_{j,j}}}=2-2\textbf{x}^T L^{-1/2}AR^{-1/2} \textbf{y}.\]
This implies that $\textbf{x}^T L^{-1/2}AR^{-1/2} \textbf{y}\le 1$ for all unit vectors $\textbf{x}$ and $\textbf{y}$, which yields $\|L^{-1/2}AR^{-1/2}\|\le 1$. It remains to show that $\|L^{-1/2}AR^{-1/2}\|\ge 1$. Indeed, by definition of $L^{-1/2}AR^{-1/2}$, we have $\text{\textbf{\textit{u}}}_1^T L^{-1/2}AR^{-1/2} \text{\textbf{\textit{v}}}_1=1$. On the other hand, by the second part of the lemma, $\text{\textbf{\textit{u}}}_1=L^{-1/2}AR^{-1/2} \text{\textbf{\textit{v}}}_1$, hence
\[
\text{\textbf{\textit{v}}}_1^T \left(L^{-1/2}AR^{-1/2}\right)^T L^{-1/2}AR^{-1/2} \text{\textbf{\textit{v}}}_1=\text{\textbf{\textit{u}}}_1^T L^{-1/2}AR^{-1/2} \text{\textbf{\textit{v}}}_1=1.
\]
Therefore, the first part of the lemma and the definition of the operator norm give $\|{L^{-1/2}AR^{-1/2}}\| \ge 1$.
\end{proof}

\begin{comment}
The following theorem proved by Thompson \cite{thompson1972principal} is useful when one wants to obtain non-trivial bounds on the singular values of submatrices.

\begin{theorem}[Interlacing Theorem for singular values]\label{thm:Interlacing-Theorem}
Let $A$
be an $m\times n$ matrix and let 
$$\alpha_1\geq\alpha_2\geq\ldots\geq\alpha_{\min\{m,n\}}$$
be its singular values. Let $B$ be any $p\times q$ submatrix of $A$ and let $$\beta_1\geq\beta_2\geq\ldots\geq\beta_{\min\{p,q\}}$$ be its singular values. Then
\[
\begin{aligned}
\alpha_i\geq\beta_i,\quad\quad\quad\quad\quad\quad\:\:&\text{for }i=1,2,\ldots,\min\{p,q\},\\
\beta_i\geq \alpha_{i+(m-p)+(n-q)},\quad&\text{for }i\leq\min\{p+q-m,p+q-n\}.
\end{aligned}
\]
\end{theorem}
\end{comment}

Finally, we consider two additional norms. The $\ell^1\to \ell^2$ operator norm of a matrix $A$ measures the maximum Euclidean norm of the output vector $Ax$ given that the input vector $x$ has $\ell^1$-norm 1. Namely,
\[\|A\|_{1 \to 2} = \max_{x \in \mathbb R^n\setminus \{0\}} \frac{\|Ax\|_2}{\|x\|_1} = \max_{x\in \mathbb R^n: \|x\|_1 = 1} \|Ax\|_2.\]

Note that the unit ball for the $\ell^1$ norm in $\mathbb R^n$ is the cross-polytope obtained as a convex hull of $\{\pm e_j\colon j\in [n]\}$. Since $x\in \mathbb R^n\mapsto \|Ax\|_2$ is a convex function, the latter maximum is attained in one of the extremal points of the unit $\ell^1$-ball, namely

\begin{equation}\label{eq:12}
\|A\|_{1 \to 2} = \max_{j\in [n]} \|(A_{i,j})_{i=1}^m\|_2.
\end{equation}

The last norm we introduce is the infinity norm of the matrix $A$, which is defined as $\|A\|_{\infty} = \max_{i,j} |A_{i,j}|$.

\subsubsection{Probabilistic matrix inequalities}
The first inequality is a ready-to-use estimate of the expected norm of a random submatrix obtained via random column extraction, due to Tropp~\cite[Theorem~23]{Tro08}.
We recall that $\text{rk}(A)$ denotes the rank of a matrix $A$ and, for a subset $I\subseteq [n]$, we denote by $P_{I,n}$ the matrix $(\mathbf{1}_{\{i=j\in I\}})_{i,j\in [n]}$.

\begin{lemma}\label{lem:compression}
Fix integers $m'\le n$, a real-valued matrix $A$ with $n$ columns, and a random set $I'$ distributed uniformly among all subsets of $[n]$ of size $m'$.
Fix also $p\ge2$ and $q=\max\{2,p/2,2\log(\mathrm{rk}(AP_{I',n}))\}$.
Then,
\begin{equation}\label{eq:compression}
        (\mathbb E \|AP_{I',n}\|^p)^{1/p}
        \le
        3\sqrt{q} (\mathbb E \|AP_{I',n}\|^p_{1\to2})^{1/p}
        +\sqrt{m'/n} \|A\|.
\end{equation}
\end{lemma}

The next lemma is an efficient decoupling result from the same work of Tropp~\cite[Theorem~9]{Tro08}.

\begin{lemma}\label{lem:decoupling}
Fix an $n\times n$ symmetric matrix $A$ with all-zero main diagonal and $M\in [n]$.
Denote by $R$ a random subset of $[n]$ of size $M$.
For every $p\ge 1$, there exists a partition $(S,T)$ of $[n]$ with $|S|=\lfloor n/2\rfloor$ and $|T|=\lceil n/2\rceil$ such that, for every $p\ge 1$,
\begin{equation}\label{eq:decoupling}
        (\mathbb E \|P_{R,n} A P_{R,n}\|^p)^{1/p}
        \le
        2\max_{m,m'} (\mathbb E
        \|P_{I,|S|} A_{S\times T} P_{I',|T|}\|^p)^{1/p},
\end{equation}
where $A_{S\times T}$ is the restriction of $A$ on the rows of $S$ and the columns of $T$, $I,I'$ are independent random subsets of $S,T$ with sizes $m,m'$, respectively, and the maximum is over $m,m'$ with $m+m'=M$, $m\le |S|$ and $m'\le |T|$.
\end{lemma}

The last preliminary lemma bounds from above the $\ell^1\to \ell^2$ operator norm of a random row-and-column extraction of a given matrix $A$. While its proof still relies on a standard argument, we did not find a precise reference for it and include it for completeness.

\begin{lemma}\label{lem:selected-column}
Fix positive integers $m,m',n,n'$, a real-valued $n\times n'$ matrix $A$, and two independent random sets $I,I'$ distributed uniformly among all subsets of $[n]$ and $[n']$ of sizes $m,m'$, respectively.
Then, for every $p\ge 2$,
\begin{equation}\label{eq:selected-column}
        (\mathbb E \|P_{I,n}AP_{I',n'}\|^p_{1\to 2})^{1/p}
        \le
        2\sqrt{m/n} \|A\|_{1\to2}
        + 4\sqrt{p/2+\log m'}\|A\|_{\infty}.
\end{equation}
\end{lemma}

\begin{proof}
Set $P_I=P_{I,n}$ and $P_{I'}=P_{I',n'}$, and denote by $A_{*,j}$ the $j$-th column of $A$. For a fixed column $j\in [n']$,
\[
        \|P_I A_{*,j}\|_2^2=\sum_{i\in I} A_{i,j}^2.
\]
Hence, $\mathbb{E}\|P_IA_{*,j}\|_2^2=\frac{m}{n}\|A_{*,j}\|_2^2$ since the random set $I$ samples $m$ values from $A_{1,j}^2,A_{2,j}^2,\ldots,A_{n,j}^2$. 
Moreover, the variance of $\|P_IA_{*,j}\|_2^2$ is equal to
\begin{align*}
\sigma^2 &= \sum_{i=1}^n \bigg(\frac{m}{n} A_{i,j}^4 - \bigg(\frac{m}{n} A_{i,j}^2\bigg)^2\bigg) + \sum_{1\le i_1<i_2\le n} 2\bigg(\frac{m(m-1)}{n(n-1)} A_{i_1,j}^2 A_{i_2,j}^2 - \bigg(\frac{m}{n}\bigg)^2 A_{i_1,j}^2 A_{i_2,j}^2\bigg)\\
&= \frac{m(n-m)}{n^2(n-1)} \sum_{1\le i_1<i_2\le n} (A_{i_1,j}^2-A_{i_2,j}^2)^2\le \frac{m}{n(n-1)} \sum_{1\le i_1<i_2\le n} 4\|A\|_\infty^2 (A_{i_1,j}-A_{i_2,j})^2\\
&\le \frac{m}{n(n-1)} \sum_{1\le i_1<i_2\le n} 4\|A\|_\infty^2 (2A_{i_1,j}^2+2A_{i_2,j}^2) = \frac{8m}{n} \|A\|_\infty^2 \|A_{*,j}\|_2^2\le \frac{8m}{n} \|A\|_\infty^2 \|A\|_{1\to 2}^2.
\end{align*}
By Hoeffding's comparison theorem (Lemma~\ref{lem:Hoeffding}) and Bernstein's inequality (Lemma~\ref{lem:Bernstein}), for any $t \ge 0$,
\[
        \mathbb P\left(\|P_IA_{*,j}\|^2_2 - m\|A_{*,j}\|^2_2/n \ge t\right)
        \le \exp\left(-\frac{t^2}{2\sigma^2 + 2\|A\|_\infty^2 t / 3}\right)\le \exp\left(-\frac{t^2}{16m\|A\|_\infty^2\|A\|_{1\to 2}^2/n + 2\|A\|_\infty^2 t / 3}\right).
\]
In particular, for every $u\ge 0$,
\begin{equation}\label{eq:column-tail^2}
\mathbb P\left(\|P_IA_{*,j}\|^2_2 - m\|A_{*,j}\|^2_{2}/n\ge 4\sqrt{mu/n} \|A\|_{\infty} \|A\|_{1\to 2} + 4\|A\|^2_{\infty} u\right) \le \e^{-u},
\end{equation}
and using that $\|A_{*,j}\|_2 \le \|A\|_{1\to 2}$ allows us to transform \eqref{eq:column-tail^2} to
\begin{equation}\label{eq:column-tail}
\mathbb P\left(\|P_IA_{*,j}\|_2\ge \sqrt{m/n} \|A\|_{1\to 2} + 2\sqrt{u} \|A\|_{\infty} \right) \le \e^{-u}.
\end{equation}
Next, condition on $I'$ and recall that $|I'|=m'$. 
By~\eqref{eq:12} and a union bound over the $m'$ elements in $I'$,
\begin{equation}\label{eq:bdu}
\mathbb P(\|P_IAP_{I'}\|_{1\to 2} \ge \sqrt{m/n}\|A\|_{1\to 2} + 2\sqrt{u+\log m'}\|A\|_{\infty} \mid I') \le m'\e^{-u-\log m'} = \e^{-u}. 
\end{equation}
To bound the left hand side in~\eqref{eq:selected-column}, set $t(0) \coloneqq (\sqrt{m/n}\|A\|_{1\to 2} + 2\sqrt{\log m'}\|A\|_{\infty})^{p}$ and note that
\[\mathbb E \|P_IAP_{I'}\|^p_{1\to 2} = \mathbb E\bigg[\int_{0}^{\infty} \mathbb P(\|P_IAP_{I'}\|^p_{1\to 2}\ge t\mid I') \mathrm{d} t\bigg] \le t(0) + \mathbb E\bigg[\int_{t(0)}^{\infty} \mathbb P(\|P_IAP_{I'}\|^p_{1\to 2}\ge t\mid I') \mathrm{d} t\bigg].\]
By performing the change of variables $t(u) = (\sqrt{m/n}\|A\|_{1\to 2} + 2\sqrt{u+\log m'}\|A\|_{\infty})^p$ and using~\eqref{eq:bdu}, 
\begin{align*}
\mathbb E\bigg[\int_{u=0}^{\infty} \mathbb P\bigg(\|P_IAP_{I'}\|^p_{1\to 2}\ge \left(\sqrt{m/n}\|A\|_{1\to 2} + 2\sqrt{u+\log m'}\|A\|_{\infty}\right)^{p}\mid I'\bigg) t'(u) \mathrm{d} u\bigg]\le \int_{u=0}^{\infty} t'(u) \e^{-u}\mathrm{d} u.
\end{align*}
To evaluate the remaining integral, we use integration by parts. For an exponential random variable $U \sim \text{Exp}(1)$,
\begin{align*}
\int_{u=0}^{\infty} t'(u) \e^{-u} \mathrm{d} u = \left[ t(u) \e^{-u} \right]_{0}^\infty + \int_{0}^{\infty} t(u) \e^{-u} \mathrm{d}u\le \mathbb{E}[t(U)],
\end{align*}
where we ignored the first term and extended the range of integration in the second term in the middle expression to derive the inequality.
By Minkowski's inequality applied to the $\ell^p$-norm and then to the $\ell^{p/2}$-norm, we obtain
\[
\mathbb{E}[t(U)]^{1/p} \le \sqrt{m/n}\|A\|_{1\to 2} + \mathbb E[(U+\log m')^{p/2}]^{1/p}\cdot 2\|A\|_{\infty} \le  \sqrt{m/n}\|A\|_{1\to 2} +  \sqrt{\mathbb E[U^{p/2}]^{2/p}+\log m'}\cdot 2\|A\|_{\infty}.\]
Using that the $r$-norm of an $\mathrm{Exp}(1)$ random variable is equal to $\Gamma(r+1)$ and the Gamma function satisfies $\Gamma(r+1)^{1/r} \le r$ for all $r \ge 1$, we obtain
\[
\mathbb E \|P_IAP_{I'}\|^p_{1\to 2}\le t(0) + \mathbb E[t(U)]\le t(0) + (\sqrt{m/n}\|A\|_{1\to 2} + 2\sqrt{p/2+\log m'} \|A\|_{\infty})^p.
\]
Finally, taking the $p$-th root of both sides and using the inequality $(x+y)^{1/p} \le x^{1/p} + y^{1/p}$, for any $x,y\ge 0$, shows that
\[(\mathbb E \|P_IAP_{I'}\|^p_{1\to 2})^{1/p}\le 2\sqrt{m/n}\|A\|_{1\to 2} + 4\sqrt{p/2+\log m'} \|A\|_{\infty},\]
as desired.
\end{proof}

We are ready to state and prove our main technical tool in the proof of Theorem~\ref{thm:patched1}, which sharpens Corollary~6.8 from~\cite{FHMV25}.

\begin{proposition}\label{prop:fixed68}
Fix a real symmetric $n\times n$ matrix $A$ and a uniformly random subset $R\subseteq [n]$ of size $M\in [n]$. Then, for every $p\ge 2$ and $q=\max\{p,2\log M\}$,
\begin{equation}\label{eq:fixed68}
        (\mathbb E\|P_RAP_R\|^p)^{1/p}
        \le
        24\sqrt{2Mq/n} \|A\|_{1\to 2}
        +50q\|A\|_{\infty} + 4M\|A\|/n.
\end{equation}
\end{proposition}

\begin{proof}
First, suppose that the main diagonal of $A$ contains only zeros. By Lemma~\ref{lem:decoupling}, it suffices to bound the $p$-norm of $\|P_IA_{S\times T}P_{I'}\|$ where $P_I=P_{I,|S|}$ and $P_{I'}=P_{I',|T|}$ for any partition $(S,T)$ of $[n]$, 
positive integers $m,m'$ with $m+m'=M$, $m\le |S|$, $m'\le |T| $ and independent subsets $I\subseteq S$ and $I'\subseteq T$ with sizes $m,m'$, respectively. 
%Define $\rho \coloneqq m/|S|$ and $\rho'\coloneqq m'/|T|$.

Next, condition on $I$ and apply Lemma~\ref{lem:compression} on the column restriction $I'$. 
Since the submatrix $P_{I} A_{S\times T} P_{I'}$ has rank at most $\min\{m,m'\}$, we have $2\log(\text{rk}(P_I A_{S\times T} P_{I'})) \le 2\log M \le q$ and therefore, conditionally on $I$,
\begin{equation}\label{eq:proof-first}
        Z\coloneq \mathbb E [\|P_IA_{S\times T}P_{I'}\|^p\mid I]^{1/p}
        \le
        3\sqrt q\cdot \mathbb E
        [\|P_{I}A_{S\times T} P_{I'}\|^p_{1\to2}\mid I]^{1/p}
        +\sqrt{m'/|T|}\cdot \|P_{I}A_{S\times T}\|.
\end{equation}
By Lemma~\ref{lem:selected-column} and the inequalities $p\le q$ and $\log m' \le \log M \le q/2$, we have $p/2 + \log m' \le q$. Thus,
\begin{equation}\label{eq:proof-col}
        (\mathbb E \|P_{I,n}AP_{I',n'}\|^p_{1\to 2})^{1/p}
        \le
        2\sqrt{m/|S|}\cdot \|P_I A_{S\times T}\|_{1\to2}
        + 4\sqrt{q}\|P_I A_{S\times T}\|_\infty \le 2\sqrt{m/|S|}\cdot \|A\|_{1\to 2} + 4\sqrt{q}\|A\|_{\infty}.
\end{equation}
Thus, taking $p$-norms on both sides in~\eqref{eq:proof-first} (seen as $I$-measurable random variables) and using~\eqref{eq:proof-col} together with Minkowski's inequality for the right hand side yields
\begin{equation}\label{eq:pnorm}
\mathbb E[\|P_IA_{S\times T}P_{I'}\|^p]^{1/p} = \mathbb E[Z^p]^{1/p}\le
        3\sqrt{q} (2\sqrt{m/|S|}\cdot\|A\|_{1\to 2} + 4\sqrt{q}\|A\|_{\infty})
        +\sqrt{m'/|T|}\cdot \mathbb E[\|P_{I}A_{S\times T}\|^p]^{1/p},
\end{equation}
where the first expectation is with respect to $I,I'$ while the second and the third expectations are only with respect to $I$.
For the last expectation in \eqref{eq:pnorm}, we apply Lemma~\ref{lem:compression} to $A_{S\times T}^T$ and the column restriction $I$. 
By recalling that $P_I$ is symmetric and thus $\|P_IA_{S\times T}\|=\|A_{S\times T}^T P_I\|$ thanks to Lemma~\ref{lem:operator}(i), we have that
\begin{equation}\label{eq:proof-row}
\mathbb E[\|P_{I}A_{S\times T}\|^p]^{1/p} = \mathbb E [\|A_{S\times T}^T P_{I}\|^p]^{1/p}\le
        3\sqrt q\cdot \mathbb E[\|A_{S\times T}^T P_I\|^p_{1\to 2}]^{1/p}
        +\sqrt{m/|S|}\cdot\|A_{S\times T}^T\|.
\end{equation}
Furthermore, by using Lemma~\ref{lem:operator}(i) for the matrix $A$ and Lemma~\ref{lem:selected-column} with $P_{I,n}=I_{m'}$ and $P_{I',n}=P_I$, we have
\begin{equation}\label{eq:proof-row2}
\begin{split}
3\sqrt q\cdot \mathbb E[\|A_{S\times T}^T P_I\|^p_{1\to 2}]^{1/p}
        &+\sqrt{m/|S|}\cdot\|A_{S\times T}^T\|\\
        &\le 
        3\sqrt q\cdot (2\|A_{S\times T}^T\|_{1\to 2}+4\sqrt{p/2+\log m}\cdot \|A_{S\times T}^T\|_{\infty})+\sqrt{m/|S|}\cdot\|A_{S\times T}^T\|.
\end{split}
\end{equation}
Combining \eqref{eq:pnorm}, \eqref{eq:proof-row}, \eqref{eq:proof-row2} and using that $\|A_{S\times T}^T\|_{\infty}\le \|A\|_{\infty}$, that $\|A_{S\times T}^T\|_{1\to 2}\le \|A\|_{1\to 2}$ since the matrix $A$ is symmetric, and that $\|A_{S\times T}^T\|\le \|A\|$ by Lemma~\ref{lem:operator}(ii), we obtain
\begin{equation}\label{eq:combined-raw}
\begin{split}
\mathbb E[\|P_IA_{S\times T}P_{I'}\|^p]^{1/p} 
\le \Big(6&\sqrt{qm/|S|} + 6\sqrt{qm'/|T|}\Big)\cdot\|A\|_{1\to 2}\\
&+ \Big(12q+12\sqrt{q(m'/|T|)(p/2+\log m)}\Big) \|A\|_{\infty} + \sqrt{mm'/|S||T|}\cdot\|A\|.
\end{split}
\end{equation}
To finish the proof for the all-zero diagonal case, note that $(m'/|T|)(p/2+\log m)\le q$ as well as
\[\sqrt{m/|S|}+\sqrt{m'/|T|}\le 2\sqrt{M/(n/2)} = 2\sqrt{2M/n}\quad\text{and}\quad \sqrt{mm'/|S||T|}\le M/(n/2) = 2M/n.\]
Therefore,
\begin{equation}
        (\mathbb E\|P_RA P_R\|^p)^{1/p}
        \le
        2\left(12\sqrt{2Mq/n} \|A\|_{1\to 2}
        +24q\|A\|_{\infty} + 2M\|A\|/n\right).
\end{equation}

We now consider the case where $A$ may have non-zero entries on its diagonal. Write $A_0=A-D$, where $D$ is the diagonal $n\times n$ matrix with entries $A_{1,1},\ldots,A_{n,n}$. 
By the previous case applied to $A_0$, we have
\begin{equation}
        (\mathbb E\|P_RA_0 P_R\|^p)^{1/p}
        \le
        24\sqrt{2Mq/n} \|A_0\|_{1\to 2}
        +48q\|A_0\|_{\infty} + 4M\|A_0\|/n.
\end{equation}
Now, note that $\|{A_0}\| \le \|{A}\|+\|{D}\|$, $\|{A_0}\|_{1\rightarrow 2} \le \|{A}\|_{1\rightarrow 2}$ by~\eqref{eq:12}, $\|{A_0}\|_\infty \le \|{A}\|_\infty$ and 
$\|{P_JAP_J}\| \le \|{P_JA_0P_J}\|+\|{P_JDP_J}\|\le \|{P_JA_0P_J}\|+\|{D}\|$ by Lemma~\ref{lem:operator}(ii). By Minkowski's inequality, this implies
\[
    (\mathbb E\|P_RA P_R\|^p)^{1/p} \le (\mathbb E\|P_R A_0 P_R\|^p)^{1/p} + \|{D}\|
    \le 24\sqrt{2Mq/n} \|A\|_{1\to 2}+48q\|A\|_{\infty} + 4M(\|A\|+\|D\|)/n+ \|{D}\|.
\]
Notice that $\|{D}\|=\max_{i\in [n]}|{A_{i,i}}| \le \|{A}\|_\infty$ to complete the proof.
\end{proof}

\subsubsection{\texorpdfstring{Proof of Theorem~\ref{thm:patched1}}{Proof of Theorem~4.16}}

First, we establish the regularity of the percolated subgraph $H$.
By using Lemma~\ref{lem:chernoff_hyp} and a union bound, we have
\begin{align*}
\mathbb{P}(\exists v\in X: \deg_H(v)\notin(1\pm2\gamma)\sigma d) 
&\le\sum_{v\in V(G)}\mathbb{P}(v\in X)\mathbb{P}(\deg_H(v)\notin(1\pm2\gamma)\sigma d\mid v\in X)\\
&\le 2m\exp\bigg(-\frac{(\gamma \sigma d)^2}{3\sigma d+\gamma \sigma d}\bigg)\le 2m\exp(-\gamma^2 \sigma d/4),
\end{align*}
where in the second inequality we used suboptimal constants in the denominator to account for the fact that $v$ is already included in $X$, and thus we deal with a random subset of size $m-1$ of a set of size $n-1$ (which are both suitably large).
Under the degree hypothesis $\sigma d\ge C\gamma^{-2}\log m$, fixing $C$ sufficiently large yields failure probability at most $m^{-1/100}/2$.

Next, denote by $d_1,\ldots,d_n$ the degrees of the vertices in $G$, 
by $\mathbbm{1}_n$ the all-1 column vector of length $n$, 
and denote by $D$ the corresponding diagonal matrix of dimensions $n\times n$. Also, denote by $A_G,A_H$ the adjacency matrices of $G,H$.
To bound the second singular value $s_2(A_H)$, we would like to discard the largest eigenvalue of $A_H$ and instead have $s_2(A_H)$ as an operator norm of some related matrix. 
To this end, define
\begin{equation*}
B=D^{-1/2}A_GD^{-1/2}-a^{-1} (D^{1/2}\mathbbm{1}_n)\cdot (\mathbbm{1}_n^T D^{1/2})
        \qquad\text{where}\qquad
        a=d_1+d_2+\ldots+d_n.
\end{equation*}
Note that the entries of the matrix $B$ are given by
\begin{equation}\label{eq:B}
B_{ij} = \frac{A_{ij}}{\sqrt{d_i d_j}} - \frac{\sqrt{d_i d_j}}{a},
\end{equation}
providing the following information for the norms of $B$:
\begin{enumerate}[label={(\alph*)},ref=\alph*]
    \item\label{item:1} for the $\ell_2\to \ell_2$ operator norm: by Lemma~\ref{lem:low-rank-approximation} and Lemma~\ref{obs: |N|} we have
    %since $D^{1/2}w$ is equal to the principal eigenvector of $D^{-1/2}A_GD^{-1/2}$ with principal eigenvalue 1, 
    %the rank-one projection (second term in $B$) eliminates this eigenvalue and therefore, since all entries of $A_G$ and $D$ are non-negative,
    \[\|B\|=s_2(D^{-1/2}A_GD^{-1/2})\le s_2(A_G)\|D^{-1/2}\|^2_{\infty} \le (1+\gamma)\lambda /d.\]
    \item for the infinity norm: since $A_{ij} \in \{0, 1\}$, $d_i = (1\pm \gamma)d$ and $a = (1\pm \gamma) nd$, using~\eqref{eq:B}, it can be readily checked that the maximum absolute value of any entry is at most $(1+\gamma)^2/d$.
    \item for the $\ell_1 \to \ell_2$ operator norm: the triangle inequality yields
    \begin{equation}    \label{eq: B12}
        \|B\|_{1\rightarrow 2} 
        \le \|{D^{-1/2}A_GD^{-1/2}}\|_{1\rightarrow 2} + a^{-1} \|(D^{1/2}w)\cdot (w^T D^{1/2})\|_{1\rightarrow 2}.        
    \end{equation}
    Let us bound each of the terms appearing on the right hand side. First, 
    \[
    \|{D^{-1/2}A_GD^{-1/2}}\|_{1\rightarrow 2}
    \le \|{D^{-1/2}}\|^2_{\infty} \|{A_G}\|_{1\rightarrow 2} = \|{D^{-1}}\|_{\infty} \|{A_G}\|_{1\rightarrow 2}.
    \]
    We have $\|{D^{-1}}\|_{\infty} = \max_{i\in [n]} (1/d_i) \le (1+\gamma)^2/d$ 
    and $\|{A_G}\|_{1\rightarrow 2} = \max_{i\in [n]} \sqrt{d_i} \le (1+\gamma)\sqrt{d}$. Thus, 
    \begin{equation}    \label{eq: AGbar}
        \|{D^{-1/2}A_GD^{-1/2}}\|_{1\rightarrow 2} \le (1+\gamma)^3/\sqrt{d},   
    \end{equation}
    and also
    \begin{equation}    \label{eq: D11D 12}
        \|{D^{1/2} \mathbbm{1}_n \mathbbm{1}_n^T D^{1/2}}\|_{1\rightarrow 2}
        \le \|{D}\|_{\infty} \cdot \|{\mathbbm{1}_n \mathbbm{1}_n^T}\|_{1\rightarrow 2}
        \le (1+\gamma)^2d \cdot \sqrt{n}.   
    \end{equation}
    Combining \eqref{eq: B12}, \eqref{eq: AGbar}, \eqref{eq: D11D 12} and the inequality $a \ge (1-\gamma)dn$, we obtain
    \begin{equation}    \label{eq: B12 bounded}
        \|B\|_{1\rightarrow 2} 
        \le \frac{(1+\gamma)^3}{\sqrt{d}} + \frac{1}{(1-\gamma)dn} \cdot (1+\gamma)^2d\sqrt{n}
        \le \frac{2.5}{\sqrt{d}}.
    \end{equation}
\end{enumerate}
Set $q = \max\{p,2\log m\}$. 
Then, by Proposition~\ref{prop:fixed68},
\[
\begin{split}
        \mathbb{E}[\|P_XBP_X\|^p]^{1/p}
        &\le
        24\sqrt{2mq/n}\|B\|_{1\to 2}+50q\|B\|_{\infty}+4m\|B\|/n\\
        &\le 60\sqrt{2mq/nd}+50q(1+\gamma)^2/d+4(1+\gamma)m\lambda/nd \eqqcolon \Psi
\end{split}
\]
%Multiplying by $D^{1/2}$ on the left and on the right and using the non-negativity of the entries of all matrices shows that
%\begin{equation}\label{eq:D-half-bound}
%        \mathbb{E}[\| D^{1/2}P_XB P_XD^{1/2}\|^p]^{1/p}\le \mathbb{E}[\|P_XBP_X\|^p]^{1/p}\cdot \min\left\{\sqrt{d_i}:i\in [n]\right\}^2
%        \le
%        (1+\gamma)\Psi d.
%\end{equation}
Now, a closer look at the matrix $P_XB P_X$ shows that actually
\begin{align*}
P_XB P_X 
&= P_X D^{-1/2} A_G D^{-1/2} P_X - a^{-1} (P_X D^{1/2} \mathbbm{1}_n)\cdot (\mathbbm{1}_n^T D^{1/2} P_X)\\ 
&= D^{-1/2} A_H D^{-1/2} - a^{-1} (P_X D^{1/2} \mathbbm{1}_n)\cdot (\mathbbm{1}_n^T D^{1/2} P_X).
\end{align*}
Now, using that $(P_X D^{1/2} \mathbbm{1}_n)\cdot (\mathbbm{1}_n^T D^{1/2} P_X)$ is a matrix of rank one and $P_X D^{1/2} \mathbbm{1}_n$ is a common eigenvector of each of $D^{-1/2} A_H D^{-1/2}$ and $(P_X D^{1/2} \mathbbm{1}_n)\cdot (\mathbbm{1}_n^T D^{1/2} P_X)$, we have
\[s_2(D^{-1/2} A_H D^{-1/2})\le \|D^{-1/2} A_H D^{-1/2} - a^{-1} (P_X D^{1/2} \mathbbm{1}_n)\cdot (\mathbbm{1}_n^T D^{1/2} P_X)\| = \|P_XB P_X\|,\]
and therefore 
\[s_2(A_H)\le s_2(D^{-1/2} A_H D^{-1/2})\cdot \|D^{1/2}\|^2_{\infty} \le (1+\gamma) d \|P_XBP_X\| \implies \mathbb E[s_2(A_H)^p]^{1/p}\le (1+\gamma) d \Psi.\]

Next, choosing $p=K\log m$ for a suitably large constant $K$ forces $q = p$. 
Applying Markov's inequality to the $p$-th moments then yields
\[
        \mathbb{P}(s_2(A_H)\ge 2(1+\gamma) d \Psi)
        \le 2^{-p}
        \le m^{-1/100}/2.
\]
Consequently, on the complimentary high-probability event, we have
\begin{equation}\label{eq:s2-before-absorb}
        s_2(A_H)
        \le 2(1+\gamma) d \Psi
        \le 120(1+\gamma)\sqrt{2mqd/n}+100q(1+\gamma)^3+8(1+\gamma)^2m\lambda/n.
\end{equation}
To show that the right-hand side is strictly bounded by $6\sigma\lambda$, we analyse each term separately.
First, since $m = \lfloor \sigma n \rfloor \le \sigma n$, the last term satisfies
\[
8(1+\gamma)^2 m\lambda/n \le 8(1+\gamma)^2\sigma\lambda.
\]
Given the parameter restriction $\gamma \in (0, 1/200]$, this term is at most $8(1.005)^2\sigma\lambda\le 8.1\sigma\lambda$.

Next, recalling~\eqref{eq:patched-hypotheses} and that $q = p = K\log m$, the first two terms in~\eqref{eq:s2-before-absorb} are dominated by 
\begin{equation}\label{eq:6050}
120\sqrt{2K}(1+\gamma)\sqrt{\sigma d\log m}\qquad \text{and}\qquad 100K(1+\gamma)^3\log m,    
\end{equation}
respectively. 
From the assumed spectral hypothesis $\sigma\lambda \ge C\sqrt{\sigma d\log m}$, we obtain the relation
\begin{equation}\label{eq:sigmalog}
\sqrt{\sigma d\log m} \le \sigma\lambda/C \qquad \text{and}\qquad \log m\le (\lambda/C^2 d)\cdot \sigma\lambda\le \sigma \lambda/C^2.
\end{equation}
Consequently, the terms in~\eqref{eq:6050} satisfy
\[
120(1+\gamma)\sqrt{2mqd/n} + 100q(1+\gamma)^3 \le \left( \frac{120\sqrt{2K}(1+\gamma)}{C} + \frac{100K(1+\gamma)^3}{C^2} \right)\sigma\lambda.
\]
By choosing the implicit absolute constant $C$ to be sufficiently large relative to $K$, this combined error contribution can be made arbitrarily small.
In particular, we can bound it from above by $1.95\sigma\lambda$, and substituting this back into~\eqref{eq:s2-before-absorb} shows that
\[
s_2(A_H) \le 8.1\sigma\lambda + 1.95\sigma\lambda \le 11\sigma\lambda.
\]
Finally, observe that since $A_H$ is an adjacency matrix of a graph, it is a block diagonal matrix, thus by Lemma~\ref{lem:operator}(iii),(iv), $s_2(A_H)=\max\{|\lambda_2(A_H)|,|\lambda_n(A_H)|\}$, which completes the proof.

\section{Concluding Remarks}\label{s: conclusion}

In this paper, we proved that every $(n,d,\lambda)$-graph, with $\lambda\le c_1d$ for a sufficiently small constant $c_1>0$, contains an induced cycle of length $\Theta(n\log(d/\lambda)/d)$. 

We conclude with several directions for future research. First of all, while our bound is optimal up to the value of the implicit constant, understanding the leading term more precisely remains an interesting challenge. In particular, one could try to compare the length of a longest induced cycle with the trivial benchmark of twice the size of a maximum independent set.
Second, whether our result holds for other classic (and potentially more general) families of graphs, such as jumbled graphs, remains open. Finally, in the spirit of~\cite{AKS07, arXiv:2608.06358, GH24, JKW12, M19}, it would be very interesting to understand whether every tree with bounded maximum degree and size $\Theta(n\log(d/\lambda)/d)$ can be embedded as an induced subgraph of an $(n,d,\lambda)$-graph. 
In a related direction, one could ask for the universality of (dense) pseudorandom graphs for all graphs on a given number of vertices.

\paragraph{Acknowledgements} The authors would like to thank Anastasia Kireeva for useful comments and suggestion on Section \ref{s: generalisation of FHMV}.

\bibliographystyle{abbrv}
\bibliography{Bib}

\end{document}